\documentclass[11pt]{amsart}
\usepackage[dvips]{color}
\usepackage{amsmath}
\usepackage{amsxtra}
\usepackage{amscd}
\usepackage{amsthm}
\usepackage{amsfonts}
\usepackage{amssymb}
\usepackage{eucal}
\usepackage{epsfig}
\theoremstyle{plain}
\newtheorem{theorem}{Theorem}[section]

\newtheorem{thm}[theorem]{Theorem}
\newtheorem{cor}[theorem]{Corollary}
\newtheorem{prop}[theorem]{Proposition}
\newtheorem{lem}[theorem]{Lemma}

\newtheorem{defi}[theorem]{Definition}
\theoremstyle{definition}

\newtheorem{example}{Example}
\newcommand{\Glie}{\mathfrak{g}}             %% Lie algebra of finite type
\newcommand{\Gaff}{\widehat{\mathfrak{g}}}   %% affine Lie algebra
\newcommand{\Gafft}{\widehat{\mathfrak{g}'}} %% affine Lie algebra transposed
\newcommand{\Uc}{\mathcal{U}}                %% quantum loop algebra
\newcommand{\ev}{\textrm{ev}}                %% evaluation  
\newcommand{\Dc}{\mathcal{D}}                %% Drinfeld quantum double 
\newcommand{\CF}{\mathcal{F}}                %% category of finite-dimensional modules
\newcommand{\CR}{\textbf{R}}                 %% classification set

\newcommand{\BC}{\mathbb{C}}            %% complex numbers
\newcommand{\BZ}{\mathbb{Z}}            %% natural numbers   

\newcommand{\BV}{\textbf{V}}             %% natural representation
\newcommand{\BQ}{\textbf{Q}}             %% root lattice 
\newcommand{\BP}{\textbf{P}}             %% weight lattice 
\newcommand{\End}{\textrm{End}}          %% Endomorphism algebra
\newcommand{\Id}{\textrm{Id}}            %% Identity
\newcommand{\CB}{\mathcal{B}}            %% Young tableaux
\newcommand{\Sm}{\mathbb{S}}             %% antipode of a Hopf superalgebra
\newcommand{\Rc}{\mathcal{R}}            %% normalized R matrix

\newcommand{\super}{\mathbb{Z}_2}        %% super
\newcommand{\even}{\overline{0}}         %% even
\newcommand{\odd}{\overline{1}}          %% odd

\newtheorem{rem}[theorem]{Remark}
\makeatletter
\def\ExtendSymbol#1#2#3#4#5{\ext@arrow 0099{\arrowfill@#1#2#3}{#4}{#5}}
\def\RightExtendSymbol#1#2#3#4#5{\ext@arrow 0359{\arrowfill@#1#2#3}{#4}{#5}}
\def\LeftExtendSymbol#1#2#3#4#5{\ext@arrow 6095{\arrowfill@#1#2#3}{#4}{#5}}
\makeatother

\allowdisplaybreaks                %% equations allowed to be breaked

\begin{document}
\begin{title}[Quantum affine superalgebras]
{RTT realization of Quantum affine superalgebras and tensor products}
\end{title}
\author{Huafeng Zhang}
\address{Universit{\'e} Paris Diderot - Paris 7,  Institut de Math{\'e}matiques de Jussieu - Paris Rive Gauche CNRS UMR 7586, B\^{a}timent Sophie Germain, Case 7012, 75025 Paris Cedex 13, France}
\email{huafeng.zhang@imj-prg.fr}
\begin{abstract}
We use the RTT realization of the quantum affine superalgebra associated with the Lie superalgebra $\mathfrak{gl}(M,N)$ to study its finite-dimensional representations and their tensor products. In the case $\mathfrak{gl}(1,1)$,  the cyclicity condition of tensor products of finite-dimensional simple modules is determined completely in terms of zeros and poles of rational functions. This in turn induces cyclicity of some particular tensor products of Kirillov-Reshetikhin modules related to $\mathfrak{gl}(M,N)$.

\end{abstract}

\maketitle
\setcounter{tocdepth}{1}
%\tableofcontents
\section{Introduction}
Let $q$ be a non-zero complex number which is not a root of unity.
Let $\Glie := \mathfrak{gl}(M,N)$ be the {\it general linear Lie superalgebra}. Let $U_q(\Gaff)$ be the associated quantum affine superalgebra. This is a Hopf superalgebra neither commutative nor co-commutative, and it can be seen as a deformation of the universal enveloping algebra of the following affine Lie superalgebra $L \Glie := \Glie \otimes \BC[t,t^{-1}]$.
In this paper we are mainly concerned with the structure of tensor products of finite-dimensional simple $U_q(\Gaff)$-modules.
\subsection{Backgrounds.} Quantum superalgebras appear as the algebraic supersymmetries of some solvable models. For example, in the supersymmetric $t-J$ model,  the quantum affine superalgebra $U_q(\widehat{\mathfrak{sl}(M,N)})$ together with its highest weight Fock representations \cite{Kojima} is used to diagonalize the commuting transfer matrices and to find the spectra of this integrable model. 
Another main interest in quantum superalgebras comes from the integrability structure in the context of the AdS/CFT correspondence \cite{Beisert1}. In this case, based on the simple Lie superalgebra $\mathfrak{psl}(2,2)$, several quantum superalgebras have been built as algebraic supersymmetries: the quantum deformation of extended $\mathfrak{sl}(2,2)$ in \cite{Beisert2}, the quantum affine deformation of extended $\mathfrak{sl}(2,2)$ in \cite{Beisert3}, and the conventional Yangian of extended $\mathfrak{sl}(2,2)$ in \cite{Beisert4,Beisert5}, to name a few. Representations of these superalgebras have been considered from different perspectives: \cite{Beisert6,Molev1} for centrally extended $\mathfrak{sl}(2,2)$, \cite{ADT} for the conventional Yangian, and \cite{Beisert2,Beisert3} for the quantum (affine) superalgebra of extended $\mathfrak{sl}(2,2)$.

More closely related to our present paper is the work of Bazhanov-Tsuboi \cite{BT} on Baxter's $\textbf{Q}$-operators related to the quantum affine superalgebra $U_q(\widehat{\mathfrak{sl}(2,1)})$. In {\it loc. cit} they constructed the so-called {\it oscillator representations} of the upper Borel subalgebra $\mathfrak{B}_+$. These representations gave rise directly to the $\textbf{Q}$-operators and therefore found remarkable applications in spin chain models and in quantum field theory. Their oscillation construction has been generalized to the quantum affine superalgebra $U_q(\widehat{\mathfrak{gl}(M,N)})$ in \cite{Tsuboi}. 

On the other hand, Hernandez-Jimbo \cite{HJ} constructed similar oscillator representations of the upper Borel subalgebra $\mathfrak{B}_+$ of an arbitrary non-twisted quantum affine algebra. In their context, oscillator representations were realized as certain asymptotic limits of Kirillov-Reshetikhin modules over the quantum affine algebra. The asymptotic construction enabled Frenkel-Hernandez \cite{FH} to give a representation theoretic interpretation of Baxter's $\textbf{T-Q}$ relations and to solve a conjecture of Frenkel-Reshetikhin on the spectra of quantum integrable systems \cite{FR}.

%Based on the above progress, it is natural to consider representation theory of the quantum affine superalgebra $U_q(\Gaff)$, and more specifically the quantum superalgebras related to centrally extended $\mathfrak{sl}(2,2)$. In the present paper,  $U_q(\Gaff)$ is our main concern. 

We are motivated by the following question: can the oscillator representations related to the quantum affine superalgebra $U_q(\Gaff)$ in \cite{BT,Tsuboi} be realized as asymptotic limits of Kirillov-Reshetikhin modules in the spirit of Hernandez-Jimbo? 

In the super case, the representation theory of quantum affine superalgebras is still less developed, compared to the vast literature on representations of quantum affine algebras (see the two review papers \cite{CH,L}). Let $I_0 := \{1,2,\cdots,M+N-1\}$.
 In a recent paper \cite{Z}, it was shown that, up to tensor product by one-dimensional modules, finite-dimensional simple $U_q(\Gaff)$-modules are parametrized by $I_0$-tuple of rational functions $\underline{f} =  (f_i(z))_{i\in I_0} \in \BC(z)^{I_0}$ such that:
\begin{itemize}
\item[(a)] if $i \neq M$ then there exists a polynomial $P_i(z)$ with constant term $1$ such that $f_i(z) = q_i^{\deg P_i} \frac{P_i(zq_i^{-1})}{P_i(zq_i)}$. Here $q_i = q$ for $i \leq M$ and $q^{-1}$ otherwise;  
\item[(b)] if $i = M$, then $f_i(z)$ as a meromorphic function is regular at $z = 0$ and $z = \infty$. Moreover, $f_i(0) f_i(\infty) = 1$.
\end{itemize}
Let $S(\underline{f})$ be the simple module associated to $\underline{f}$. In analogy with the non-graded case, Kirillov-Reshetikhin modules for $U_q(\Gaff)$ will be those modules $S(\varpi_{n,a}^{(i)})$ where $i \in I_0$ is a Dynkin node, $a \in \BC^{\times}$ a spectral parameter, $n \in \BZ_{>0}$ a positive integer, and $\varpi_{n,a}^{(i)}$  the $I_0$-tuple of rational functions whose $i$-th coordinate is $q_i^n\frac{1-zaq_i^{-n}}{1-zaq_i^n}$ and whose other coordinates are $1$. When $n = 1$, the Kirillov-Reshetikhin modules are called {\it fundamental modules}.

 Let us fix $i \in I_0$ and $a \in \BC^{\times}$. For $n \in \BZ_{> 0}$, the $i$-th coordinate for $\varpi_{n,aq_i^n}^{(i)}$ has the asymptotic expression $q_i^n \frac{1-za}{1-zaq_i^{2n}}$.
Informally, by taking asymptotic limit of the Kirillov-Reshetikhin modules $S(\varpi_{n,aq_i^n}^{(i)})$ we should get a \lq\lq module\rq\rq\ where the $i$-th coordinate is $1-za$ (by first forgetting the constant term $q_i^n$ and then taking the analysis limit $\lim_{n \rightarrow \infty} q_i^n = 0$).
This module should be an oscillator module. 

In the non-graded case, the asymptotic limits above come essentially from inductive systems of Kirillov-Reshetikhin modules constructed by Hernandez-Jimbo \cite{HJ}.

We would like to generalize the construction of Hernandez-Jimbo to the super case. A careful look at their construction indicates that the following cyclicity result should be a crucial step to establish inductive systems of Kirillov-Reshetikhin modules:
\begin{itemize}
\item[(C)] the tensor products $\bigotimes_{l=1}^k S(\varpi_{1,aq_j^{-2l}}^{(j)})$ for $k \geq 1$ are of highest $\ell$-weight with respect to the Drinfeld type triangular decomposition of $U_q(\Gaff)$.
\end{itemize}

Again in the non-graded case, (C) is a direct corollary of more general cyclicity results on tensor products of fundamental modules over quantum affine algebras, established by Kashiwara et al. \cite{AK,Kashiwara} using crystal bases, by Varagnolo-Vasserot \cite{VV} using quiver geometry, and by Chari \cite{Chari} using Weyl/braid group actions. 

Crystal base theory and quiver geometry for quantum affine superalgebras, or even for finite type quantum superalgebras, are still less developed \cite{BKK}. The main drawback comes from the fact that the Weyl group of $\Glie$, being $\mathfrak{S}_M \times \mathfrak{S}_N$ instead of $\mathfrak{S}_{M+N}$, is too small to capture enough information on weights and linkage. Nevertheless, in the present paper, we shall prove (C) directly by using the RTT realization of the quantum affine superalgebra $U_q(\Gaff)$ and by modifying Chari's reduction argument in \cite{Chari}.
\subsection{Main results.} Our first main result of this paper is a sufficient condition for the tensor products of fundamental modules $S(\varpi_{1,a}^{(i)})$ with $i \in I_0$ fixed to be cyclic (Theorem \ref{thm: cyclicity of tensor product of KR modules}).

The idea of proof follows largely that of Chari \cite{Chari}. The RTT generators will be used to reduce modules over $U_q(\Gaff)$ to those over the $q$-Yangian $Y_q(\mathfrak{gl}(1,1))$.  Here the $q$-Yangian is an upper Borel subalgebra  generated by half of the RTT generators.

Our second main result (Theorem \ref{thm: web property}) is on  representation theory of $Y_q(\mathfrak{gl}(1,1))$.
\begin{itemize}
\item[(1)] There is a classification of finite-dimensional simple modules, up to tensor product by one-dimensional modules, in terms of highest $\ell$-weights parametrized by the set $\CR$  of such rational functions $f(z)$ that $f(0) = 1$ (hence regular at $z = 0$). Let $V(f)$ be the simple module of highest $\ell$-weight  $f$.  
\item[(2)] For $f_1,\cdots,f_k \in \CR$, the tensor product $\bigotimes_{j=1}^k V(f_j)$ is of highest $\ell$-weight (resp. of lowest $\ell$-weight) if and only if: for all $1 \leq i < j \leq k$ (resp. for all $1 \leq j < i \leq k$) the set of poles of $f_i$ does not intersect with the set of zeros of $f_j$. 
\item[(3)] The tensor product in (2) is simple if and only if it is of highest and lowest $\ell$-weight. 
\end{itemize}
We remark that (2) is not true for the quantum affine algebra $U_q(\widehat{\mathfrak{sl}_2})$, as seen in \cite{CP1,Mukhin} where the conditions for a tensor product of simple $U_q(\widehat{\mathfrak{sl}_2})$-modules to be cyclic are more sophisticated. Also, in the non-graded case due to the Weyl group action a tensor product of simple modules is of highest $\ell$-weight if and only if it is of lowest $\ell$-weight. Hence (3) is really a special feature in the super case. 

The constructions of inductive systems of Kirillov-Reshetikhin apart from Theorem \ref{thm: cyclicity of tensor product of KR modules} and their asymptotic limits are left to a separate paper \cite{Z2}. Already, the $Y_q(\mathfrak{gl}(1,1))$-modules $V(1-z)$ and $V(\frac{1}{1-z})$ can be viewed as prototypes of asymptotic modules.

At last, we would like to point out that nearly all the results in the present paper have direct analogues when replacing the quantum affine superalgebra $U_q(\Gaff)$ (or the $q$-Yangian) by the Yangian $Y(\Glie)$,  a deformation of the universal enveloping algebra of the current Lie superalgebra $\Glie \otimes \BC[t]$. The proofs of these results are essentially the same, as $Y(\Glie)$  admits a similar RTT realization \cite{Zrb2}. In \cite[Theorem 5]{Zrb1}, a similar criteria for a tensor product of finite-dimensional simple $Y(\mathfrak{gl}(1,1))$-modules to be simple was given  by Rui-Bin Zhang with a quite different approach from ours. Cyclicity of tensor products and Drinfeld realization for the Yangian were not considered there in full generality. 

 The paper is organized as follows. \S \ref{sec: preparation} collects some basic facts about highest weight representations of the finite type quantum superalgebra $\Uc_q(\Glie)$. In \S \ref{sec: definition} we study in detail the RTT realization of the quantum affine superalgebra $U_q(\Gaff)$, review the Ding-Frenkel homomorphism between Drinfeld realization and RTT realization, and give an estimation for coproduct of Drinfeld generators (Proposition \ref{prop: Drinfeld coproduct estimation}), which is used to prove Chari's lemma in the super case (Lemma \ref{lem: cyclicity Chari}). \S \ref{sec: Yangian} discusses finite-dimensional representation theory for the $q$-Yangian $Y_q(\mathfrak{gl}(1,1))$, which is used in \S \ref{sec: proof} to prove the main Theorem \ref{thm: cyclicity of tensor product of KR modules}. In \S \ref{sec: app}, we give a proof of Proposition \ref{prop: Drinfeld coproduct estimation} on coproduct of Drinfeld generators. 

\subsection*{Acknowledgments.} The author is grateful to his supervisor David Hernandez for numerous discussions, and to Paul Zinn-Justin from whom he learned the notion of a Yang-Baxter algebra in his course \lq\lq Int\'{e}grabilit\'{e} Quantique\rq\rq\ at Universit\'{e}\ Paris 6.  

Part of the present work was done while the author was visiting Centre de Recherches Math\'{e}matiques in Montr\'{e}al and was participating in the workshops \lq\lq Combinatorial Representation Theory\rq\rq\ in Montr\'{e}al and \lq\lq Yangians and Quantum Loop Algebras\rq\rq\ in Austin. He is grateful to the organizers for hospitality and to Vyjayanthi Chari, Sachin Gautam and Valerio Toledano Laredo for stimulating discussions.

\section{Preliminaries}  \label{sec: preparation}
In this section, we introduce the basic notations concerning the finite type quantum superalgebra $\Uc_q(\mathfrak{gl}(M,N))$ and its representations. Following Benkart-Kang-Kashiwara we review the character formula for fundamental representations.

 Let $V = V_{\even} \oplus V_{\odd}$ be a vector superspace. We write $|x| = i$ for $i \in \super$ and $x \in V_i$. If $V = \oplus_{\alpha \in P} (V)_{\alpha}$ is graded by an abelian group, we write $|x|_P = \alpha$ for $\alpha \in P$ and $x \in (V)_{\alpha}$.   

Fix $M,N \in \BZ_{\geq 0}$. Set $I := \{1,2,\cdots,M+N\}$. Define two maps as follows:
\begin{displaymath}
|\cdot|: I \longrightarrow \super, i \mapsto |i| =: \begin{cases}
\even & (i\leq M),  \\
\odd & (i > M);
\end{cases} \quad d_{\cdot}: I \longrightarrow \BZ, i \mapsto d_i := \begin{cases}
1 & (i \leq M),\\
-1 & (i > M).
\end{cases}
\end{displaymath}
Set $q_i := q^{d_i}$. Set $\BP := \oplus_{i \in I} \BZ \epsilon_i$. Let $(,): \BP \times \BP \longrightarrow \BZ$ be the bilinear form defined by $(\epsilon_i,\epsilon_j) = \delta_{ij}d_i$. Let $|\cdot|: \BP \longrightarrow \super$ be the morphism of abelian groups such that $|\epsilon_i| = |i|$.

In the following, we only consider the parity $|x| \in \super$ of $x$ when either $x \in I, x \in \BP$ or $x$ is a $\super$-homogeneous vector of a vector superspace. 

If $A$ is a superalgebra, then by bracket we always mean $[a,b] := a b - (-1)^{|a||b|} ba$ for $\super$-homogeneous $a,b \in A$. Unless otherwise stated, $\Glie$ will always be the general linear Lie superalgebra $\mathfrak{gl}(M,N)$, while $\Glie' = \mathfrak{gl}(N,M)$.
Set $I_0 := I \setminus \{M+N\}$. For $i \in I_0$, set $\alpha_i := \epsilon_i - \epsilon_{i+1} \in \BP$. Introduce the root lattice $\BQ := \oplus_{i\in I_0} \BZ \alpha_i \subset \BP$. Define $\BQ_{\geq 0} := \oplus_{i\in I_0} \BZ_{\geq 0} \alpha_i$.
\subsection{The quantum superalgebra $\Uc_q(\Glie)$.} This is the superalgebra defined by generators $e_i^{\pm}, t_j^{\pm 1}$ ($i \in I_0,\ j \in I$) with  $\super$-degrees $|e_i^{\pm}| = |\alpha_i|$ and $|t_j^{\pm 1}| = 0$ and with relations
\begin{align*}
& t_j t_j^{-1} = t_j^{-1} t_j = 1,\quad t_i t_j = t_j t_i \quad (i,j \in I),    \\
& t_i e_j^{\pm} t_i^{-1} = q^{\pm d_i (\epsilon_i, \epsilon_j - \epsilon_{j+1})} e_j^{\pm} \quad (i \in I, j \in I_0),   \\
& [e_i^+,e_j^-] = \delta_{ij} \frac{t_i^{d_i}t_{i+1}^{-d_{i+1}}- t_i^{-d_i}t_{i+1}^{d_{i+1}} }{q_i - q_i^{-1}} \quad (i,j \in I_0),
\end{align*}
together with Serre relations which we do not repeat (see \cite[\S 2.2]{Z} and its references). $\Uc_q(\Glie)$ has a Hopf superalgebra structure with coproduct: for $i \in I_0, j \in I$
\begin{align*}
& \Delta (e_i^+) = 1 \otimes e_i^+ + e_i^+ \otimes t_i^{-d_i}t_{i+1}^{d_{i+1}},\quad \Delta(e_i^-) = t_i^{d_i}t_{i+1}^{-d_{i+1}} \otimes e_i^- + e_i^- \otimes 1, \quad \Delta(t_j) = t_j \otimes t_j. 
\end{align*}
There exists a $\BQ$-grading on $\Uc_q(\Glie)$ respecting the Hopf superalgebra structure:
\begin{displaymath}
|t_j|_{\BQ} = 0, \quad |e_i^{\pm}|_{\BQ} = \pm \alpha_i \quad (i \in I_0,\ j \in I).
\end{displaymath}
In general, for $\alpha \in \BQ$, we have
\begin{displaymath}
(\Uc_q(\Glie))_{\alpha} = \{x \in \Uc_q(\Glie)\ |\ t_i x t_i^{-1} = q^{d_i(\epsilon_i,\alpha)} x\ \textrm{for}\ i \in I  \}.
\end{displaymath}
This $\BQ$-grading is compatible with the $\super$-grading: $(\Uc_q(\Glie))_{\alpha} \subset \Uc_q(\Glie)_{|\alpha|}$ for $\alpha \in \BQ$. 

\subsection{Highest weight representations.} Let $\lambda \in \BP$. Up to isomorphism, there exists a unique simple $\Uc_q(\Glie)$-module, denoted by $L(\lambda)$, which is generated by a vector $v_{\lambda}$ satisfying:
\begin{displaymath}
|v_{\lambda}| = |\lambda|,\quad e_i^+ v_{\lambda} = 0, \quad t_j v_{\lambda} = q^{d_j(\epsilon_j,\lambda)} v_{\lambda} \quad (i \in I_0, j \in I).
\end{displaymath}
The action of the $t_j$ endows  $L(\lambda)$ with the following $\BP$-grading: for $\mu \in \BP$
\begin{displaymath}
(L(\lambda))_{\mu} := \{x \in L(\lambda)\ |\ t_j x = q^{d_j(\epsilon_j,\mu)} x\ \textrm{for}\ j \in I\},\quad L(\lambda) = \bigoplus_{\mu \in \BP} (L(\lambda))_{\mu}.
\end{displaymath}
%It was shown \cite{Zrb} that for $\lambda \in \BP$, $L(\lambda)$ is finite-dimensional if and only if $d_i(\lambda,\alpha_i) > 0$ for all $i \in I_0\setminus \{M\}$.

\begin{example}   \label{example: natural representation quantum superalgebra}
(Natural representation.) Let $\BV = \oplus_{i\in I} \BC v_i$ be the vector superspace with $\super$-grading $|v_i| = |i|$. On $\BV$ there is a natural representation $\rho_{(0)}$ of $\Uc_q(\Glie)$ defined by: 
\begin{displaymath}
\rho_{(0)}(e_i^+) = E_{i,i+1},\quad \rho_{(0)}(e_i^-) = E_{i+1,i},\quad \rho_{(0)}(t_j) = \sum_{k\in I} q^{d_j(\epsilon_j,\epsilon_k)} E_{kk} \quad (i \in I_0, j \in I).
\end{displaymath}
Here the $E_{ij} \in \End(\BV)$ for $i,j \in I$ are defined by $E_{ij} v_k = \delta_{jk}v_i$. Clearly, $\BV = L(\epsilon_1)$ as a $\Uc_q(\Glie)$-module with $v_1$ a highest weight vector, and $(\BV)_{\epsilon_i} = \BC v_i$ for $i \in I$. 
Consider the tensor product $\BV^{\otimes 2}$ as a $\Uc_q(\Glie)$-module. Define subspaces
\begin{align}
& \BV^{+} := \bigoplus_{1 \leq i < j \leq M+N} \BC(qv_i \otimes v_j + (-1)^{|i||j|} v_j \otimes v_i) \oplus \bigoplus_{k=1}^M \BC (v_k \otimes v_k),     \label{def: quantum symmetric space}    \\
& \BV^{-} := \bigoplus_{1 \leq i < j \leq M+N} \BC(q^{-1}v_i \otimes v_j - (-1)^{|i||j|} v_j \otimes v_i) \oplus \bigoplus_{k=1}^N \BC (v_{M+k} \otimes v_{M+k}).    \label{def: quantum anti-symmetric space}
\end{align}
Then $\BV^{\otimes 2} = \BV^{+} \oplus \BV^-$ is a decomposition of $\Uc_q(\Glie)$-modules as follows:
\begin{displaymath}
L(\epsilon_1)^{\otimes 2} = L(2 \epsilon_1) \oplus L(\epsilon_1 + \epsilon_2).
\end{displaymath} 
\end{example}
 Let $V$ be a finite-dimensional $\Uc_q(\Glie)$-module $\BP$-graded via the action of the $t_i$ in the same way as $L(\lambda)$ above.
Define its {\it character} to be
\begin{equation}   \label{def: classical character}
\chi (V) := \sum_{\mu \in \BP} \dim (V)_{\mu} [\mu] \in \BZ[\BP].
\end{equation}
By definition, it is clear that $(\Uc_q(\Glie))_{\alpha} (V)_{\mu} \subseteq (V)_{\alpha+\mu}$ for  $\alpha \in \BQ,\mu \in \BP$.

 For $r \in I_0$, define the $r$-th {\it fundamental weight} $\varpi_r := \begin{cases}
\sum_{i=1}^r \epsilon_i & (r \leq M),  \\
-\sum_{i=r+1}^{M+N} \epsilon_i & (r > M).
\end{cases}$
When $1 \leq r \leq M$, let $\CB_r$ be the set of functions $f: \{1,2,\cdots, r \} \longrightarrow I$ such that: $f(i) \leq f(i')$ for $i \leq i'$; if $1 \leq i < r$ and $f(i) \leq M$, then $f(i) < f(i+1)$. ($\CB_r$ is the set of semi-standard Young tableaux of the Young diagram with $r$ rows and one column; see \cite[\S 4.1]{BKK}.)
\begin{theorem} \label{thm: BKK Schur-Weyl duality}   \cite{BKK}
For $1 \leq r \leq M$  we have
\begin{equation}   \label{equ: character of simple modules}
\chi (L(\varpi_r)) = \sum_{f \in \CB_r} [\sum_{i=1}^r \epsilon_{f(i)}] \in \BZ[\BP].
\end{equation} 
\end{theorem}
\section{Quantum affine superalgebras and their representations} \label{sec: definition}
In this section, we present in a uniform way the RTT realization of the quantum affine superalgebra $U_q(\Gaff)$ following \cite{Ding-Frenkel,FM,Molev,Tsuboi,Zyz}. Then we study highest $\ell$-weight representations.
\subsection{The quantum affine superalgebra $U_q(\Gaff)$.} \label{sec: quantum affine superalgebra}
Let $P^+,P^-$ be the projections with respect to the decomposition $\BV^{\otimes 2} = \BV^+ \oplus \BV^-$ in Example \ref{example: natural representation quantum superalgebra}. Let $c_{\BV,\BV} \in \End \BV^{\otimes 2}$ be the signed permutation $v_i \otimes v_j \mapsto (-1)^{|i||j|} v_j \otimes v_i$. Define the {\it Perk-Schultz matrix} \cite{Perk-Schultz} 
\begin{equation}    \label{equ: Perk-Schultz matrix}
R(z,w) = c_{\BV,\BV}((z q - w q^{-1}) P^+ + (w q - z q^{-1}) P^-) \in (\End \BV^{\otimes 2})[z,w]
\end{equation}
\begin{defi}      \label{def: quantum affine superalgebras}
The quantum affine superalgebra $U_q(\Gaff)$ is the superalgebra defined by 
\begin{itemize}
\item[(R1)] RTT-generators $s_{ij}^{(n)}, t_{ij}^{(n)}$ for $i,j \in I$ and $n \in \BZ_{\geq 0}$;
\item[(R2)] $\super$-grading $|s_{ij}^{(n)}| = |t_{ij}^{(n)}| = |i| + |j|$; 
\item[(R3)] RTT-relations \cite{FRT2} in $U_q(\Gaff) \otimes (\End \BV^{\otimes 2})[[z,z^{-1},w,w^{-1}]]$
\begin{eqnarray}
&& R_{23}(z,w) T_{12}(z) T_{13}(w) = T_{13}(w) T_{12}(z) R_{23}(z,w), \label{rel: RTT = TTR}  \\
&& R_{23}(z,w) S_{12}(z) S_{13}(w) = S_{13}(w) S_{12}(z) R_{23}(z,w),  \label{rel: RSS = SSR}  \\
&& R_{23}(z,w) T_{12}(z) S_{13}(w) = S_{13}(w) T_{12}(z) R_{23}(z,w),  \label{rel: RTS = STR}    \\
&& t_{ij}^{(0)} = s_{ji}^{(0)} = 0 \quad \mathrm{for}\ 1 \leq i < j \leq M+N,  \label{rel: FRTS zero condition}  \\
&& t_{ii}^{(0)} s_{ii}^{(0)} = 1 = s_{ii}^{(0)} t_{ii}^{(0)} \quad \mathrm{for}\ i \in I.   \label{rel: FRTS invertibility condition}
\end{eqnarray}
\end{itemize}
Here $T(z) = \sum_{i,j \in I} t_{ij}(z) \otimes E_{ij} \in (U_q(\Gaff) \otimes \End \BV)[[z^{-1}]]$ and $t_{ij}(z) = \sum_{n \in \BZ_{\geq 0}} t_{ij}^{(n)} z^{-n} \in U_q(\Gaff)[[z^{-1}]]$ (similar convention for $S(z)$ with the $z^{-n}$ replaced by the $z^{n}$). The $q$-Yangian $Y_q(\Glie)$ is the subalgebra of $U_q(\Gaff)$ generated by the $(s_{ii}^{(0)})^{-1}, s_{ij}^{(n)}$ for $i,j \in I$ and $n \in \BZ_{\geq 0}$.
\end{defi}  
We use the following convention. Let $A,B,C$ be three superalgebras and let $T = \sum_i a_i \otimes b_i \in A \otimes B$. Then the $T_{kl}$ with $1 \leq k < l \leq 3$ are defined by $T_{12} := \sum_i a_i \otimes b_i \otimes 1 \in A \otimes B \otimes C, T_{13} := \sum_{i} a_i \otimes 1 \otimes b_i \in A \otimes C \otimes B$ and $T_{23} := \sum_i 1 \otimes a_i \otimes b_i \in C \otimes A \otimes B$. It is also natural to define the $T_{kl}$ with $1 \leq k < l \leq n$ by introducing $n-3$ more auxiliary superalgebras. (We shall not use the $T_{kl}$ with $k > l$.) 

 $U_q(\Gaff)$ is endowed with a super bialgebra structure (set $\epsilon_{ijk} := (-1)^{(|i|+|k|)(|k|+|j|)}$):
\begin{eqnarray}   
&&\Delta (s_{ij}^{(n)}) = \sum_{a=0}^n \sum_{k \in I} \epsilon_{ijk}  s_{ik}^{(a)} \otimes s_{kj}^{(n-a)}, \quad  \Delta (t_{ij}^{(n)}) = \sum_{a=0}^n \sum_{k \in I} \epsilon_{ijk} t_{ik}^{(a)} \otimes t_{kj}^{(n-a)}. \label{for: coproduct for quantum affine superalgebra S}   
\end{eqnarray}
The antipode $\Sm: U_q(\Gaff) \longrightarrow U_q(\Gaff)$ is given by the following equations in $(U_q(\Gaff)\otimes \End \BV)[[z^{\pm 1}]]$
\begin{eqnarray}
&&(\Sm \otimes \Id)(S(z)) = S(z)^{-1}, \quad (\Sm \otimes \Id)(T(z)) = T(z)^{-1}.    \label{for: antipode for S}  
\end{eqnarray}
Here the RHS of the above formulas are well defined thanks to Relation \eqref{rel: FRTS zero condition}. It follows that $U_q(\Gaff)$ is a Hopf superalgebra, with $Y_q(\Glie)$ a sub-Hopf-superalgebra. 

In the following, we review several fundamental symmetry properties of $U_q(\Gaff)$.

There exists a $\BZ$-grading on $U_q(\Gaff)$: $|s_{ij}^{(n)}|_{\BZ} = n,\ |t_{ij}^{(n)}|_{\BZ} = -n$ afforded by the one-parameter family $(\Phi_a: a \in \BC^{\times})$ of Hopf superalgebra automorphisms:
\begin{equation}    \label{for: automorphisms of Z-graded superalgebras}
\Phi_a: U_q(\Gaff) \longrightarrow U_q(\Gaff),\quad s_{ij}^{(n)} \mapsto a^n s_{ij}^{(n)},\ t_{ij}^{(n)} \mapsto a^{-n} t_{ij}^{(n)}.
\end{equation}  

The $\BQ$-grading comes from the conjugate action of the $s_{ii}^{(0)}$: for $\alpha \in \BQ$
\begin{displaymath}
(U_q(\Gaff))_{\alpha} = \{ x \in U_q(\Gaff) \ |\ s_{ii}^{(0)} x (s_{ii}^{(0)})^{-1} = q^{(\epsilon_i,\alpha)} x \quad \textrm{for}\ i \in I   \}.
\end{displaymath}
The $\BQ$-grading is compatible with the Hopf superalgebra structure and with the $\super$-grading:
\begin{equation}     \label{equ: weight grading}
|s_{ij}^{(n)}|_{\BQ} = |t_{ij}^{(n)}|_{\BQ} = \epsilon_i - \epsilon_j \quad (i,j \in I).
\end{equation}
\subsubsection{The Perk-Schultz $R$-matrix.}  \label{sec: Perk-Schultz matrix}
 The exact form of $R(z,w) \in (\End \BV \otimes \End \BV)[z,w]$ is:
\begin{eqnarray}  \label{for: Perk-Schultz matrix coefficients}
\begin{array}{rcl}
R(z,w) &=&  \sum\limits_{i\in I}(zq_i - wq_i^{-1}) E_{ii} \otimes E_{ii}  + (z-w) \sum\limits_{i \neq j} E_{ii} \otimes E_{jj} \\
&\ & + z \sum\limits_{i<j} (q_i-q_i^{-1}) E_{ji} \otimes E_{ij} + w \sum\limits_{i<j}(q_j-q_j^{-1})  E_{ij} \otimes E_{ji}.
\end{array}  
\end{eqnarray}
Notably it satisfies the following {\it quantum Yang-Baxter equation} as observed in \cite{Perk-Schultz}
\begin{displaymath}
R_{12}(z_1,z_2)R_{13}(z_1,z_3)R_{23}(z_2,z_3) = R_{23}(z_2,z_3)R_{13}(z_1,z_3)R_{12}(z_1,z_2)
\end{displaymath} 
Let $R = R(1,0),R' = c_{\BV,\BV} R^{-1} c_{\BV,\BV}$. Then $R(z,w) = z R - w R'$ and $R' = R - (q - q^{-1}) c_{\BV,\BV}$.

Let us define the quantum affine superalgebra $U_q(\Gafft) = U_q(\widehat{\mathfrak{gl}(N,M)})$ in exactly the same way as $U_q(\Gaff)$, except that we interchange $M,N$ everywhere. Let $s_{ij}'^{(n)},t_{ij}'^{(n)}$ for $i,j \in I$ and $n \in \BZ_{\geq 0}$ be the corresponding RTT generators of $U_q(\Gafft)$, so that their $\super$-degrees are $|s_{ij}'^{(n)}| = |t_{ij}'^{(n)}| = |i|' + |j|'$ where $|i|' = \even$ for $1 \leq i \leq N$ and $\odd$ otherwise. For $i,j \in I$, set 
\begin{displaymath}
\varepsilon_{ij} := (-1)^{|i|(|i|+|j|)},\quad \varepsilon_{ij}' := (-1)^{|i|'(|i|'+|j|')}, \quad \overline{i} := M+N+1-i.
\end{displaymath}
 \begin{prop}   \label{prop: from gl(N,M) to gl(M,N)}
There are isomorphisms of Hopf superalgebras $f: U_q(\Gafft)^{\mathrm{cop}} \longrightarrow U_q(\Gaff)$ and $\Psi: U_q(\Gaff) \longrightarrow U_q(\Gaff)^{\mathrm{cop}}$ defined by
\begin{align*}
& f(s_{ij}'^{(n)}) = \varepsilon_{ji}' s_{\overline{j}\overline{i}}^{(n)},\quad  f(t_{ij}'^{(n)}) = \varepsilon_{ji}' t_{\overline{j}\overline{i}}^{(n)}, \\
& \Psi(s_{ij}^{(n)}) = \varepsilon_{ji}t_{ji}^{(n)},\quad \Psi(t_{ij}^{(n)}) = \varepsilon_{ji}s_{ji}^{(n)}.
\end{align*}
\end{prop} 
Here, for $(A,\Delta,\varepsilon)$ a Hopf superalgebra, $(A^{\mathrm{cop}},\Delta^{\mathrm{cop}},\varepsilon)$ denotes another Hopf superalgebra with the same underlying superalgebra $A$ but with twisted coproduct $\Delta^{\mathrm{cop}} := c_{A,A} \Delta$.  
\subsubsection{Evaluation morphisms.}  \label{sec: evaluation maps}
 As in Definition \ref{def: quantum affine superalgebras}, let us define $U_q(\Glie)$ to be the superalgebra generated by $s_{ij},t_{ji}$ for $1 \leq i \leq j \leq M+N$, with $\super$-degrees 
\begin{displaymath}
|s_{ij}| = |t_{ji}| = |i|+|j|
\end{displaymath}
and with RTT relations (\cite{FRT}) in $U_q(\Glie) \otimes \End \BV^{\otimes 2}$
\begin{eqnarray*}
&& R_{23}T_{12}T_{13} = T_{13}T_{12}R_{23},\quad R_{23}S_{12}S_{13} = S_{13}S_{12}R_{23}  \\
&& R_{23}T_{12}S_{13} = S_{13}T_{12}R_{23},\quad  s_{ii} t_{ii} = 1 = t_{ii}s_{ii}.
\end{eqnarray*}
Here, as usual, $T = \sum_{i\leq j} t_{ji} \otimes E_{ji},\ S = \sum_{i\leq j}s_{ij} \otimes E_{ij} \in U_q(\Glie) \otimes \End \BV$.  $U_q(\Glie)$ is endowed with a Hopf superalgebra structure with similar coproduct as in Equation \eqref{for: coproduct for quantum affine superalgebra S}.
\begin{prop}   \label{prop: evaluation morphism}
(1) The assignment $s_{ij} \mapsto s_{ij}^{(0)},\ t_{ji} \mapsto t_{ji}^{(0)}$ extends uniquely to a Hopf superalgebra morphism $\iota: U_q(\Glie) \longrightarrow U_q(\Gaff)$.

(2) The assignment $s_{ij}(z) \mapsto s_{ij} - z t_{ij},\ t_{ij}(z) \mapsto t_{ij} - z^{-1}s_{ij}$ extends uniquely to a superalgebra morphism $\ev: U_q(\Gaff) \longrightarrow U_q(\Glie)$.

(3) There is an isomorphism of Hopf superalgebras $DF: \Uc_q(\Glie) \longrightarrow U_q(\Glie)$
\begin{displaymath}
e_i^+ \mapsto \frac{s_{ii}^{-1}s_{i,i+1}}{1 - q_i^{-2}},\quad e_i^- \mapsto \frac{t_{i+1,i}t_{ii}^{-1}}{1-q_i^2},\quad t_j^{d_j} \mapsto s_{jj} = t_{jj}^{-1} \quad (i \in I_0,\ j \in I).
\end{displaymath}   
\end{prop}
Here $s_{ji} = t_{ij} = 0$ in $U_q(\Glie)$ for $i < j$. The morphism $\ev$ is called an {\it evaluation morphism}. It is clear that $\ev \circ \iota = \Id_{U_q(\Glie)}$, from which we obtain a $\BQ$-grading on $U_q(\Glie)$. 
\subsubsection{Quantum double construction.}\label{sec: quantum double}
 We reformulate the RTT definition of $U_q(\Gaff)$ as a quantum double construction, as in the non-graded case \cite[Theorem 16]{FRT}. This will in turn give an RTT presentation of the $q$-Yangian $Y_q(\Glie)$ in Definition \ref{def: quantum affine superalgebras}.

Let $A,B$ be two Hopf superalgebras. Call a bilinear form $\varphi: A \times B \longrightarrow \BC$ a {\it Hopf pairing} if $\varphi$ is of $\super$-degree $\even$, and if $\varphi$ satisfies 
\begin{eqnarray*}
&&\varphi(a,bb') = (-1)^{|b||b'|} \varphi(a_{(1)},b) \varphi(a_{(2)},b'),\quad \varphi(a,1) = \varepsilon_A(a);   \\
&&\varphi(aa',b) = \varphi(a',b_{(1)}) \varphi(a,b_{(2)}),\quad \varphi(1,b) = \varepsilon_B(b)
\end{eqnarray*}
for $\super$-homogeneous $a,a' \in A$ and $b,b' \in B$. One can endow the vector superspace $A \otimes B$ with a unique Hopf superalgebra structure satisfying \cite[Theorem 3.2]{Rosso}
\begin{itemize}
\item[(QD1)] $a \mapsto a \otimes 1,\ b \mapsto 1 \otimes b$ are morphisms of Hopf superalgebras respectively;
\item[(QD2)] for $\super$-homogeneous $a \in A, b \in B$, we have $(a \otimes 1)(1 \otimes b) = a \otimes b$ and
\begin{equation*}
(1 \otimes b) (a \otimes 1) = (-1)^{|a_{(1)}||b| + (|b_{(2)}|+|b_{(3)}|)|a_{(2)}| + |a_{(3)}||b_{(3)}|} \varphi (a_{(1)}, S_B (b_{(1)})) \varphi(a_{(3)}, b_{(3)}) a_{(2)} \otimes b_{(2)}.
\end{equation*}
\end{itemize}
Let $\Dc_{\varphi}(A,B)$ be the Hopf superalgebra thus obtained.

In our context, $A$ (resp. $B$) is the superalgebra generated by the $s_{ij}^{(n)},(s_{ii}^{(0)})^{-1}$ (resp. the $t_{ij}^{(n)},(t_{ii}^{(0)})^{-1}$) with $\super$-gradings and with defining relations as in Definition \ref{def: quantum affine superalgebras} (without Relation \eqref{rel: RTS = STR} which makes no sense). Clearly $A$ and $B$ are Hopf superalgebras with coproducts defined by Equation \eqref{for: coproduct for quantum affine superalgebra S}.
\begin{prop}   \label{prop: quantum double construction}
There exists uniquely a Hopf pairing $\varphi: A \times B \longrightarrow \BC$ such that
\begin{equation}  \label{equ: Hopf pairing}
\sum_{i,j,a,b \in I} E_{ab} \otimes E_{ij} \sum_{m,n \in \BZ_{\geq 0}} z^{-m}w^n \varphi(s_{ij}^{(n)},t_{ab}^{(m)}) = \frac{R(z,w)}{zq-wq^{-1}} \in (\End \BV)^{\otimes 2}[[z^{-1},w]].
\end{equation}
The assignment $s_{ij}^{(n)} \otimes 1 \mapsto s_{ij}^{(n)},\ 1 \otimes t_{ij}^{(n)} \mapsto t_{ij}^{(n)}$ extends uniquely to a surjective morphism of Hopf superalgebras $D: \Dc_{\varphi}(A,B) \longrightarrow U_q(\Gaff)$ whose kernel is the ideal generated by the
\begin{displaymath}
s_{ii}^{(0)} \otimes 1 - 1 \otimes (t_{ii}^{(0)})^{-1},\quad 1 \otimes t_{ii}^{(0)} - (s_{ii}^{(0)})^{-1} \otimes 1 \quad (i \in I).
\end{displaymath} 
Moreover, $D$ restricts to a Hopf superalgebra isomorphism $D|_A: A \longrightarrow Y_q(\Glie)$. 
\end{prop} 
\noindent
\textit{Sketch of Proof.} The proof is almost identical to that in non-graded case \cite{FRT}. For completeness we indicate the main idea of proof. By abuse of language, let $\mathcal{F}_A$ (resp. $\mathcal{F}_B$) be the superalgebra freely generated by the $s_{ij}^{(n)}$ (resp. the $t_{ij}^{(n)}$) for $i,j \in I, n \in \BZ_{\geq 0}$, and with $\super$-gradings $|s_{ij}^{(n)}| = |t_{ij}^{(n)}| = |i|+|j|$. Then $\mathcal{F}_A$ and $\mathcal{F}_B$ are super bialgebras with coproduct given by Equation \eqref{for: coproduct for quantum affine superalgebra S}. Now Formula \eqref{equ: Hopf pairing} above determines a bilinear form $\varphi: \mathcal{F}_A \times \mathcal{F}_B \longrightarrow \BC$ satisfying all the properties of a Hopf pairing. According to \cite[Chapter 3]{Rosso} it is enough to show that $\varphi$ respects Relations \eqref{rel: RTT = TTR}-\eqref{rel: RSS = SSR}, \eqref{rel: FRTS zero condition}-\eqref{rel: FRTS invertibility condition}, and that (QD2) is equivalent to Relation \eqref{rel: RTS = STR}. We only check Relation \eqref{rel: RSS = SSR}. (The other relations can be done in the same way.) For this, define the bilinear map
\begin{align*}
& \begin{cases}
\varphi_3: \mathcal{F}_A^{\otimes 2} \otimes (\End \BV)^{\otimes 2} \times \mathcal{F}_B^{\otimes 2}  \otimes \End \BV \longrightarrow (\End \BV)^{\otimes 3}  \\
(a \otimes a' \otimes x \otimes y, b \otimes b' \otimes z) \mapsto (-1)^{(|x|+|y|)(|b|+|b'|+|z|) + |a'||b|} \varphi(a,b)\varphi(a',b')z \otimes x \otimes y.
\end{cases}  
\end{align*}
for $\super$-homogeneous vectors $a,a',x,y,z,b,b'$. Then Relation \eqref{rel: RSS = SSR} amounts to:
\begin{align*}
& \varphi_3(R_{34}(z,w)S_{13}(z)S_{24}(w)- S_{14}(w)S_{23}(z)R_{34}(z,w), T_{23}(u)T_{13}(u)) = 0.
\end{align*}
From the definitions of $\varphi$ and $\varphi_3$, we see that the LHS of the above equation becomes:
\begin{displaymath}
(uq-wq^{-1})^{-1}(uq-zq^{-1})^{-1} \times (R_{23}(z,w) R_{13}(u,w) R_{12}(u,z) - R_{12}(u,z) R_{13}(u,w) R_{23}(z,w)),
\end{displaymath}
which is zero because of the Yang-Baxter equation in \S \ref{sec: Perk-Schultz matrix}.  \hfill $\Box$

It is interesting to find the Casimir element (if it exists) with respect to this Hopf pairing; this would be the universal $R$-matrix \cite{Damiani} for $U_q(\Gaff)$. 
\subsubsection{Ding-Frenkel homomorphism.} \label{sec: Drinfeld}
The Gauss decomposition gives uniquely 
\begin{displaymath}
e_{ij}^{\pm}(z),\ f_{ji}^{\pm}(z),\ K_l^{\pm}(z) \in U_q(\Gaff)[[z^{\pm 1}]] \quad \textrm{for}\ 1 \leq i < j \leq M+N,\ 1 \leq l \leq M+N
\end{displaymath}
such that in the superalgebra $(U_q(\Gaff) \otimes \End \BV)[[z,z^{-1}]]$
\begin{equation*} 
\begin{cases}
S(z) = (\sum\limits_{i<j} f_{ji}^+(z) \otimes E_{ji} + 1 \otimes \Id_{\BV}) (\sum\limits_l K_l^+(z) \otimes E_{ll} )(\sum\limits_{i<j} e_{ij}^+(z) \otimes E_{ij} + 1 \otimes \Id_{\BV}), \\
T(z) = (\sum\limits_{i<j} f_{ji}^-(z) \otimes E_{ji} + 1 \otimes \Id_{\BV}) (\sum\limits_l K_l^-(z) \otimes E_{ll} )(\sum\limits_{i<j} e_{ij}^-(z) \otimes E_{ij} + 1 \otimes \Id_{\BV}).
\end{cases}  
\end{equation*}
For example, $K_1^+(z) = s_{11}(z)$ and $K_1^-(z) = t_{11}(z)$. For $i \in I_0 = I \setminus \{M+N\}$, define
\begin{displaymath}
X_i^+(z) = e_{i,i+1}^+(z) - e_{i,i+1}^-(z) = \sum_n X_{i,n}^+ z^n,\quad X_i^-(z) = f_{i+1,i}^-(z) - f_{i+1,i}^+(z) = \sum_{n} X_{i,n}^- z^n.
\end{displaymath}
\begin{thm} \label{thm: Ding-Frenkel}   \cite{Ding-Frenkel,Zyz}
 As a superalgebra $U_q(\Gaff)$ is generated by the coefficients of $X_i^{\pm}(z),K_j^{\pm}(z)$ with $i \in I_0,j \in I$. Moreover,
\begin{align*}
& K_i^{\epsilon}(z) K_j^{\epsilon'}(w) = K_j^{\epsilon'}(w) K_i^{\epsilon}(z),  \\
& (\textrm{Cartan}) \begin{cases}
 K_i^{\epsilon}(z) X_j^{\epsilon'}(w) = X_j^{\epsilon'}(w) K_i^{\epsilon}(z) \quad \textrm{for}\ i \notin \{j,j+1\},   \\
 K_i^{\epsilon}(z) X_i^{\pm}(w) = (\frac{q_i z - q_i^{-1}w}{z-w})^{\mp 1} X_i^{\pm}(w) K_i^{\epsilon}(z), \\
 K_{i+1}^{\epsilon}(z) X_i^{\pm}(w) = (\frac{q_{i+1}^{-1} z - q_{i+1} w}{z-w})^{\mp 1} X_i^{\pm}(w) K_{i+1}^{\epsilon}(z), 
\end{cases}    \\
& (\textrm{Drinfeld}) \begin{cases}
 X_i^{\epsilon}(z) X_j^{\epsilon}(w) - X_j^{\epsilon}(w) X_i^{\epsilon}(z) = 0 \quad \textrm{if}\ |i-j| \geq 2,     \\
 X_M^{\epsilon}(z) X_M^{\epsilon}(w) + X_M^{\epsilon}(w) X_M^{\epsilon}(z) = 0,      \\
 (q_i^{\mp 1} z - q_i^{\pm 1} w ) X_i^{\pm}(z) X_i^{\pm}(w) = (q_i^{\pm 1} z - q_i^{\mp 1} w) X_i^{\pm}(w) X_i^{\pm}(z) \quad \textrm{if}\ i \neq M,   \\
 (q_{i+1}z - q_{i+1}^{-1} w) X_i^+(z) X_{i+1}^+(w) = (z-w) X_{i+1}^+(w) X_i^+(z),     \\ 
 (z-w) X_{i}^-(z) X_{i+1}^-(w) = (q_{i+1}z - q_{i+1}^{-1} w) X_{i+1}^-(w)X_i^-(z),
\end{cases}    \\
&[X_i^+(z), X_j^-(w)] = \delta_{ij}(q_i - q_i^{-1}) \delta(\frac{z}{w}) (K_{i+1}^+(z)K_i^+(z)^{-1} - K_{i+1}^-(w)K_i^-(w)^{-1}),   \\
& (\textrm{Serre}) \begin{cases}
[X_i^{\epsilon}(z_1),[X_i^{\epsilon}(z_2), X_j^{\epsilon}(w)]_q]_{q^{-1}} + \{z_1 \leftrightarrow z_2  \} = 0 \quad \textrm{if}\ (i \neq M,\ |j-i|=1), \\
[[[X_{M-1}^{\epsilon}(u),X_M^{\epsilon}(z_1)]_q, X_{M+1}^{\epsilon}(v)]_{q^{-1}}, X_M^{\epsilon}(z_2)] + \{z_1 \leftrightarrow z_2 \} = 0 \quad \textrm{if}\ (M,N > 1).
\end{cases}
\end{align*}
\end{thm}
Here $[a,b]_u := ab - (-1)^{|a||b|} u b a$ for $\super$-homogeneous $a,b \in U_q(\Gaff)$ and $u \in \BC$.

\noindent
\textit{Remark.} Let $\Rc_q(z,w)$ be $R(z,w)$ divided by $zq - w q^{-1}$. Then $\Rc_{q^{-1}}(w,z)$
is indeed the $R$-matrix used in \cite{CWWZ,Zyz} (resp. in \cite{Ding-Frenkel}) to define centrally extended quantum affine $\mathfrak{gl}(M,N)$ (resp. quantum affine $\mathfrak{gl}(M,0)$ with $N = 0$). The above theorem appeared in \cite{CWWZ,Zyz} (after trivialization of center $q^c = 1$ and minor corrections of exponents of $q$); the idea of proof is exactly the same as in the non-graded case \cite{Ding-Frenkel}.   

\subsection{Coproduct formulas.} For $s \in \BZ_{\geq 0}$ let $K_{i,\pm s}^{\pm}$ by the coefficient of $z^{\pm s}$ in the power series $K_i^{\pm}(z)$. The following proposition will be proved in Appendix \ref{sec: app}. 
\begin{prop}     \label{prop: Drinfeld coproduct estimation}
 Let $i \in I_0, j \in I, n \in \BZ, s \in \BZ_{\geq 0}$. Then 
\begin{eqnarray}
&& \Delta (K_{j,\pm s}^{\pm}) - \sum_{a=0}^s K_{j,\pm a}^{\pm} \otimes K_{j,\pm (s-a)}^{\pm} \in \sum_{\alpha \in \BQ_{\geq 0} \setminus \{0\}} (U_q(\Gaff))_{\alpha} \otimes (U_q(\Gaff))_{-\alpha},  \label{equ: coproduct H}    \\
&&\Delta (X_{i,n}^+) - 1 \otimes X_{i,n}^+ \in \sum_{\alpha \in \BQ_{\geq 0} \setminus \{ 0\}} (U_q(\Gaff))_{\alpha} \otimes (U_q(\Gaff))_{\alpha_i-\alpha},  \label{equ: coproduct X}   \\
&&\Delta (X_{i,n}^-) - X_{i,n}^- \otimes 1 \in \sum_{\alpha \in \BQ_{\geq 0} \setminus \{ 0\}} (U_q(\Gaff))_{\alpha-\alpha_i} \otimes (U_q(\Gaff))_{-\alpha}.  \label{equ: coproduct Y}
\end{eqnarray}
\end{prop}
%This proposition turns out to be the crucial step towards the multiplicative structure of Frenkel-Reshetikhin $q$-characters \cite{FR,Z2}. Actually in the special case $M=N=1$ one has exact coproduct formulas for all the Drinfeld generators \cite{Z3}.
\subsection{Highest $\ell$-weight modules.}  \label{sec: highest l weight} 
 Let $V$ be a $U_q(\Gaff)$-module. A non-zero vector $v \in V \setminus \{0\}$ is called a {\it highest $\ell$-weight vector} if $v$ is $\super$-homogeneous and 
\begin{displaymath}
s_{ij}^{(n)} v = 0 = t_{ij}^{(n)}v,\quad s_{kk}^{(n)}v,t_{kk}^{(n)}v \in \BC v \quad (n \in \BZ_{\geq 0},\ i,j,k \in I,\ i < j).  
\end{displaymath}
V is called a {\it highest $\ell$-weight module} if $V = U_q(\Gaff)v$ for some highest $\ell$-weight vector. The notions of  {\it lowest $\ell$-weight vector} and {\it lowest $\ell$-weight module} are defined by replacing the above condition $(i< j)$ with $(i > j)$. Similarly, one can define the notions of highest/lowest $\ell$-weight modules/vectors for representations of the $q$-Yangian $Y_q(\Glie)$.

Let $V, V'$ be $U_q(\Gaff)$-modules and let $v,v'$ be highest $\ell$-weight vectors in $V,V'$ respectively. Let $f_i^{\pm}(z),g_i^{\pm}(z) \in (\BC[[z^{\pm 1}]])^{\times}$ for $i \in I$ be such that
\begin{displaymath}
s_{ii}(z) v = f_i^+(z) v,\quad t_{ii}(z) v = f_i^-(z) v,\quad s_{ii}(z) v' = g_i^+(z) v',\quad t_{ii}(z) v' = g_i^-(z) v'. 
\end{displaymath}
From the Gauss decomposition in \S \ref{sec: Drinfeld}, we see that $v \otimes v'$  is of highest $\ell$-weight:
\begin{align*}
K_i^{\pm}(z) v = f_i^{\pm}(z)v,\quad K_i^{\pm}(z) v' = g_i^{\pm}(z) v', \quad K_i^{\pm}(z) (v \otimes v') = f_i^{\pm}(z)g_i^{\pm}(z) v \otimes v'.
\end{align*}
 This observation will be used in Appendix \ref{sec: app} to conclude the proof of Proposition \ref{prop: Drinfeld coproduct estimation}.

The following super analogue of Chari's lemma \cite[Lemma 4.2]{Chari} is essential in \S\S \ref{sec: Yangian}-\ref{sec: proof}.
\begin{lem}    \label{lem: cyclicity Chari}
Let $V_+$ (resp. $V_-$) be a $U_q(\Gaff)$-module of highest (resp. lowest) $\ell$-weight. Let $v_+ \in V_+$ (resp. $v_- \in V_-$) be a highest (resp. lowest) $\ell$-weight vector. Then the $U_q(\Gaff)$-module $V_+ \otimes V_-$ (resp. $V_- \otimes V_+$) is generated by $v_+ \otimes v_-$ (resp. $v_- \otimes v_+$).
\end{lem}  
\begin{proof}
This comes essentially from Proposition \ref{prop: Drinfeld coproduct estimation} on coproduct of Drinfeld generators. We only prove the first part: $V_+ \otimes V_- = U_q(\Gaff) (v_+ \otimes v_-)$, as the second part follows by using the Hopf superalgebra isomorphism $f$ in Proposition \ref{prop: from gl(N,M) to gl(M,N)}, which preserves highest/lowest $\ell$-weight properties. Since $V_-$ is a lowest $\ell$-weight $U_q(\Gaff)$-module with lowest $\ell$-weight vector $v_-$, $V_-$ is spanned as a vector superspace by the $X_{i_1,m_1}^+ X_{i_2,m_2}^+ \cdots X_{i_s,m_s}^+ v_-$ for $s \in \BZ_{\geq 0}$ and $i_t \in I_0, m_t \in \BZ$.  Hence, $V_-$ is $\BQ_{\geq 0}$-graded by the action of the $s_{ii}^{(0)}$ with $(V)_{\gamma}$ spanned by the above vectors such that $\gamma = \alpha_{i_1}+\alpha_{i_2} + \cdots + \alpha_{i_s}$. This induces a $\BZ_{\geq 0}$-grading on $V$ with $(V_-)_n$ spanned by the above vectors with $n = s$. We prove by induction on $n \in \BZ_{\geq 0}$ that 
\begin{displaymath}
(P_n):\ V_+ \otimes (V_-)_n \subseteq U_q(\Gaff)(v_+ \otimes v_-).
\end{displaymath}
When $n = 0$, $(V_-)_0 = \BC v_-$. For all $v \in V_+$, since $(U_q(\Gaff))_{-\alpha} v_- = 0$ for $\alpha \in \BQ_{\geq 0} \setminus \{0\}$, 
\begin{displaymath}
X_{i,m}^- (v \otimes v_-) = X_{i,m}^- v \otimes v_-.
\end{displaymath}
Since $V_+$ is spanned by the $X_{i_1,m_1}^- X_{i_2,m_2}^- \cdots X_{i_s,m_s}^- v_+$, we get $V_+ \otimes v_- \subseteq U_q(\Gaff)(v_+ \otimes v_-)$. Now assume $(P_k)$ for $k \leq n$. Let us prove $(P_{n+1})$. For $\super$-homogeneous vectors  $v_1 \in V_+$ and $v_2 \in (V_-)_{\beta} \subseteq (V_-)_n$ where $\beta \in \BQ_{\geq 0}$:
\begin{eqnarray*}
&&X_{i,m}^+ (v_1 \otimes v_2) \in  (-1)^{|i||v_1|} v_1 \otimes X_{i,m}^+ v_2 + \sum_{\alpha \in \BQ_{\geq 0}\setminus \{0\}} (U_q(\Gaff))_{\alpha} v_1 \otimes (U_q(\Gaff))_{\alpha_i - \alpha}v_2, \\
&&(U_q(\Gaff))_{\alpha_i - \alpha}v_2 \subseteq (V_-)_{\beta + \alpha_i - \alpha} \subseteq  \sum_{k\leq n} (V_-)_k \quad \textrm{for}\ \alpha \in \BQ_{\geq 0}\setminus \{0\}.
\end{eqnarray*}
It follows that $v_1 \otimes X_{i,m}^+ v_2 \in U_q(\Gaff) (v_+ \otimes v_-)$. As $(V_-)_{n+1}$ is spanned by the $X_{i,m}^+ v_2$ with $v_2 \in (V_-)_{n}$, we conclude.
\end{proof}
\section{Representations of the $q$-Yangian $Y_q(\mathfrak{gl}(1,1))$}  \label{sec: Yangian}
Fix $M = N = 1$ and $\Glie = \mathfrak{gl}(1,1)$. In this section we study the category $\CF$ of finite-dimensional representations of the $q$-Yangian $Y_q(\Glie)$. 
\subsection{Simple objects in $\CF$.}  \label{sec: simple}
 Let us first construct some obvious $Y_q(\Glie)$-modules.

 Let $D = \BC v$ be an one-dimensional $Y_q(\Glie)$-module. As $v$ is a highest/lowest $\ell$-weight vector, there exist $s \in \super, a,b \in \BC^{\times}$ and $f(z),g(z) \in 1 + z\BC[[z]]$ such that
\begin{displaymath}
|v| = s,\quad s_{11}(z)v = a f(z)v,\quad s_{22}(z) v = b g(z) v,\quad s_{12}(z) v = s_{21}(z) v = 0. 
\end{displaymath} 
It follows from Theorem \ref{thm: Ding-Frenkel} that $X_{1,n}^+ v = 0 = X_{1,n+1}^- v$. Henceforth $K_1^+(z)(K_2^{+}(z))^{-1} v \in \BC^{\times} v$, meaning $f(z) = g(z)$. Reciprocally, for $s \in \super, f(z) \in 1 + z\BC[[z]]$ and $a,b \in \BC^{\times}$, the above formula indeed defines a one-dimensional $Y_q(\Glie)$-module structure on $\BC v$ denoted by $\BC_{s,f,a,b}$. Define $\BC_s := \BC_{s,1,1,1}, \BC_f := \BC_{\even,f,1,1}$ and $\BC_{(a,b)} := \BC_{\even,1,a,b}$. Then as $Y_q(\Glie)$-modules, $\BC_{s,f,a,b} \cong \BC_s \otimes \BC_f \otimes \BC_{(a,b)}$.

 Following \cite{Tsuboi}, let us define $\dot{U}_q(\Glie)$ to be the superalgebra generated by $\dot{s}_{ij},\dot{t}_{ji},\dot{s}_{ii}^{-1}$ for $1 \leq i \leq j \leq 2$, with $\super$-degrees and defining relations the same as those for $U_q(\Glie)$ in \S \ref{sec: evaluation maps} except the last relation which is replaced by $$\dot{s}_{ii}\dot{s}_{ii}^{-1} = 1 = \dot{s}_{ii}^{-1} \dot{s}_{ii}.$$
$\dot{U}_q(\Glie)$ is $\BQ$-graded with respect to the conjugate actions of the $\dot{s}_{ii}$. For $1 \leq i < j \leq 2$, $|\dot{s}_{ij}|_{\BQ} = \epsilon_i - \epsilon_j = - |\dot{t}_{ji}|_{\BQ}$. Furthermore,
Let us write down the defining relations of $\dot{U}_q(\Glie)$:
\begin{align*}
& \dot{s}_{12}^2 = 0 = \dot{t}_{21}^2,\quad \dot{s}_{12}\dot{t}_{21} + \dot{t}_{21}\dot{s}_{12} = (q-q^{-1})(\dot{t}_{11}\dot{s}_{22} - \dot{s}_{11}\dot{t}_{22}).
\end{align*}
There are well-defined evaluation morphisms $\ev_a$ for $a \in \BC^{\times}$
\begin{displaymath}
\ev_a: Y_q(\Glie) \longrightarrow \dot{U}_q(\Glie),\quad s_{ij}(z) \mapsto \dot{s}_{ij} - z a \dot{t}_{ij}.
\end{displaymath}
As usual, we understand that $\dot{s}_{ji} = 0 = \dot{t}_{ij}$ when $i < j$. 
From the above presentation of $\dot{U}_q(\Glie)$ and from the evaluation morphisms, it is easy to build up explicit representations for $Y_q(\Glie)$. Let $a \in \BC^{\times}$. We shall define two evaluation representations $\rho_a^{\pm}$ of $Y_q(\Glie)$ on the vector superspace $\BV = \BC v_1 \oplus \BC v_2$. It is enough to give their generating matrices $[\rho_a^{\pm}] := (\rho_a^{\pm}(s_{ij}(z)))_{1\leq i,j \leq 2}$ with respect to the standard basis $(v_1,v_2)$. More precisely,
\begin{eqnarray*}
[\rho_a^+] &=& \begin{pmatrix}
(1-za)E_{11} + (q^{-1}-zaq) E_{22} & (q-q^{-1})E_{12} \\
-za E_{21} & E_{11}+q^{-1}E_{22}
\end{pmatrix},\\ 
\left[\rho_a^-\right] &=& \begin{pmatrix}
E_{11} + q^{-1}E_{22} & (q^{-1}-q)E_{12} \\
-zaE_{21} & (1-za)E_{11} + (q^{-1}-zaq)E_{22}
\end{pmatrix}.
\end{eqnarray*} 
Let $L_{1,a}^{\pm}$ be the $Y_q(\Glie)$-modules associated with the representations $\rho_a^{\pm}$.

Now we have a highest $\ell$-weight classification of finite-dimensional simple $Y_q(\Glie)$-modules.
\begin{lem}  \label{lem: finite-dimensional simples for gl(1,1)}
(1) A finite-dimensional simple $Y_q(\Glie)$-module must be of highest $\ell$-weight.

(2) Let $S$ be a simple $Y_q(\Glie)$-module generated by a highest $\ell$-weight vector $v$ with
\begin{displaymath}
|v| = \even,\quad s_{ii}(z) v = f_i(z) v,\quad f_i(z) \in (\BC[[z]])^{\times}\quad \textrm{for}\ i = 1,2.
\end{displaymath}
Then $S$ is finite-dimensional if and only if $\frac{f_1(z)}{f_2(z)} = \frac{P(z)}{Q(z)}$ for some polynomials $P(z),Q(z) \in \BC[z]$ with non-zero constant terms.
\end{lem}
\begin{proof}
The arguments in the proofs of \cite[Lem.4.12, Prop.6.1]{Z} apply to (1) and the \lq\lq only of\rq\rq part of (2): from the relations in Theorem \ref{thm: Ding-Frenkel} between the  $X_{1,n}^-$ and $K_{1,1}^+ = s_{11}^{(1)}$ one deduces linear relations of the $X_{1,n}^-v$ with $n > 0$; then from relations between $X_{1,0}^+$ and the $X_{1,n}^-$ one concludes that the coefficient of $v$ in $K_2^+(z)K_1^+(z)^{-1} v = s_{22}(z)s_{11}(z)^{-1} v$ is indeed a rational function. For the \lq\lq if\rq\rq part, write
\begin{displaymath}
P(z) = a \prod_{i=1}^m (1-z c_i),\quad Q(z) = b \prod_{j=1}^n (1 - z d_j),\quad c_i,d_j,a,b \in \BC^{\times}.
\end{displaymath}
Then as in the proof of \cite[Theorem 3.11]{HJ} $S$ is a sub-quotient of the tensor product
\begin{displaymath}
(\bigotimes_{i=1}^m L_{1,c_i}^+) \otimes (\bigotimes_{j=1}^n L_{1,d_j}^-) \otimes \BC_{(f_1(0),f_2(0))} \otimes \BC_{f'},
\end{displaymath}
where $f'(z) = f_1(z) f_1(0)^{-1} \prod_{i=1}^m (1-zc_i)^{-1}$. Since the $L_{1,a}^{\pm}$ are two-dimensional, $S$ must be finite-dimensional.
\end{proof}
Let us define $\CR$ to be the the set of rational functions $f(z) \in \BC(z)$ such that $f(0) = 1$. Note that $\CR$ can be viewed as a subset of $\BC[[z]]$. For $f \in \CR$, let $V(f)$ be the simple $Y_q(\Glie)$-module generated by a highest $\ell$-weight vector $v$ satisfying
\begin{displaymath}
|v| = \even,\quad s_{11}(z) v = f(z) v,\quad s_{22}(z) v = v.
\end{displaymath}
According to Lemma  \ref{lem: finite-dimensional simples for gl(1,1)}, $V(f)$ is finite-dimensional. Moreover, All finite-dimensional simple $Y_q(\Glie)$-modules can be factorized uniquely into  $V(f) \otimes D$ with $D$  one-dimensional and $f \in \CR$. For example, when $a \in \BC^{\times}$,
\begin{displaymath}
L_{1,a}^+ \cong V(1-za),\quad L_{1,a}^-  \cong V(\frac{1}{1-za}) \otimes \BC_{1-za}.
\end{displaymath}   
\subsection{Tensor product of simple modules.}For $f \in \CR$, let  $Z(f)$ (resp. $P(f)$) be the set of zeros (resp. poles) of the meromorphic function $f$.  It is possible that $\infty \in Z(f) \cup P(f)$.  The main result of this section can be stated as follows.
\begin{thm}  \label{thm: web property}
Let $f_1,f_2,\cdots,f_s \in \CR$. For $1 \leq i \leq s$, let $v_i$ be a highest $\ell$-weight vector in the simple $Y_q(\Glie)$-module $V(f_i)$. Let $V :=  \bigotimes_{i=1}^s V(f_i)$ and $v := \bigotimes_{i=1}^s v_i \in V$. Then 
\begin{itemize}
\item[(a)] $V = Y_q(\Glie) v$ if and only if $P(f_i) \cap Z(f_j) = \emptyset$ for all $1 \leq i < j \leq s$;
\item[(b)] $Y_q(\Glie)v$ is the unique simple sub-$Y_q(\Glie)$-module of $V$ if and only if $Z(f_i) \cap P(f_j) = \emptyset$ for all $1 \leq i < j \leq s$;
\item[(c)] $V$ is simple if and only if $P(f_i) \cap Z(f_j) = \emptyset$ for all $1 \leq i \neq j \leq s$. 
\end{itemize} 
\end{thm}
\begin{rem}
The theorem above can be viewed as a super analogue of \cite[Theorems 3.4,4.8]{CP1} on classification and construction of finite-dimensional simple $U_q(\widehat{\mathfrak{sl}_2})$-modules in terms of Drinfeld polynomials. See \cite[Theorem 4.6]{Mukhin} for a closer statement involving rational functions instead of Drinfeld polynomials. (c) says that tensor products of simple modules satisfy the web property, as proved in \cite{Hernandez} for quantum affine algebras.
\end{rem}
The proof of Theorem \ref{thm: web property} will be given in \S \ref{sec: proof of web}. As an important consequence, let us discuss the factorization of simple objects in $\CF$. Let $f \in \CR$. Write $f(z) = \frac{N(z)}{D(z)}$ where
\begin{displaymath}
N(z) = \prod_{i=1}^s (1 - z a_i),\quad D(z) = \prod_{i=1}^s (1 - z b_i) 
\end{displaymath} 
such that $a_i,b_i \in \BC$ and $a_i \neq b_j$ for $1 \leq i,j \leq s$ (possibly $a_i = 0$ or $b_i = 0$). Then
\begin{eqnarray*}
V(f) & \cong & \bigotimes_{i=1}^s V(\frac{1-za_i}{1-zb_i}).
\end{eqnarray*}
According to Theorem \ref{thm: web property}, these are factorizations of simple modules into prime simple modules. Here by a {\it prime} simple module we mean a simple module $S$ which can not be written as $S_1 \otimes S_2$ with $S_i$ being modules of dimension $> 1$ \cite[\S 2.2]{HL}.

We have seen in \S \ref{sec: simple} the explicit formulas for $V(1-za)$ and $V(\frac{1}{1-za})$. There still remains the third kind of prime simple modules, namely $V(\frac{1-za}{1-zb})$ for $a,b \in \BC^{\times}$ and $a \neq b$. Indeed, it is easy to check the following without using Theorem \ref{thm: web property} (b): the tensor product of highest $\ell$-weight vectors in $V(1-za) \otimes V(\frac{1}{1-zb})$ generates the unique simple sub-$Y_q(\Glie)$-module, which is two-dimensional and isomorphic to $V(\frac{1-za}{1-zb})$. Let $\rho_{a,b}$ be the corresponding representation of $Y_q(\Glie)$ on $\BV$. After some base change the generating matrix becomes
\begin{eqnarray*}
[\rho_{a,b}] &=& \begin{pmatrix}
\frac{1-za}{1-zb}E_{11} + \frac{q^{-1}-zaq}{1-zb} E_{22} & \frac{(q^{-1}-q)(b-a)}{1-zb} E_{12} \\
\frac{-z}{1-zb} E_{21} & E_{11} + \frac{q^{-1}-zbq}{1-zb} E_{22}
\end{pmatrix}.
\end{eqnarray*}
Remark that the matrix $[\rho_{a,b}]$ is well-defined even if $ab = 0$. So the $[\rho_{a,b}]$ with $a,b \in \BC$ and $a \neq b$ describe all the prime simple modules in $\CF$.

Fix $a,b \in \BC$ such that $a \neq b$.  Since $Y_q(\Glie)$ is a Hopf superalgebra, associated to the representation $\rho_{a,b}$ of $Y_q(\Glie)$ on $\BV$ is the dual representation $\rho_{a,b}^*$ on the dual space $\BV^*$ defined by: $\rho_{a,b}^* (x) := (\rho_{a,b}(\Sm x))^*\quad \textrm{for}\ x \in Y_q(\Glie)$. Let $(v_1^*,v_2^*)$ be the dual basis of $\BV^*$ with respect to $(v_1,v_2)$. Let $e_{ij} \in \End \BV^*$ be such that $e_{ij} v_k^* = \delta_{jk} v_i^*$. Then $E_{ij}^* = \varepsilon_{ij} e_{ji}$. Let us compute the generating matrix of $\rho_{a,b}^*$ with respect to the basis $(v_2^*,v_1^*)$. By definition, $[\rho_{a,b}^*]_{ij} = \rho_{a,b}(\Sm(s_{ij}(z)))^*$. In view of Equation \eqref{for: antipode for S},
\begin{displaymath}
[\rho_{a,b}(\Sm(s_{ij}(z)))]_{1 \leq i,j \leq 2}  = [\rho_{a,b}]^{-1}
\end{displaymath}
as $2\times 2$ matrices over the superalgebra $\End \BV$. A direct calculation indicates:
\begin{displaymath}
[\rho_{a,b}]^{-1} = \begin{pmatrix}
\frac{q^{-1}-zbq}{q^{-1}-zaq} E_{11} + \frac{1-zb}{q^{-1}-zaq} E_{22} & \frac{(q^{-1}-q)(a-b)}{q^{-1}-zaq} E_{12}  \\
\frac{z}{q^{-1}-zaq} E_{21} & E_{11} + \frac{1-za}{q^{-1}-zaq} E_{22}
\end{pmatrix}, \quad [\rho_{a,b}^*] = \frac{1-za}{q^{-1}-zaq} [\rho_{b,a}].
\end{displaymath}
In conclusion, as $Y_q(\Glie)$-modules:
\begin{displaymath}
V(\frac{1-za}{1-zb})^* \cong \BC_{\odd} \otimes \BC_{(q,q)} \otimes \BC_{\frac{1-za}{1-zaq^2}} \otimes V(\frac{1-zb}{1-za}). 
\end{displaymath}
\subsection{Proof of Theorem \ref{thm: web property}.} \label{sec: proof of web}
Note that (c) follows directly from (a) and (b). Also, in view of the factorization property and duality property discussed above, it is enough to prove (a) and (b) under the condition that the $f_i \in \CR$ are of the form $f_i(z) = \frac{1-za_i}{1-zb_i}$ where $a_i,b_i \in \BC$ and $a_i \neq b_i$. In this case, $P(f_i) \cap Z(f_j) = \emptyset$ if and only if $b_i \neq a_j$.  Moreover, the $V(f_i)$ are always two-dimensional, and $V(f_i)^* \cong V(f_i^{-1}) \otimes D_i$ for some one-dimensional module $D_i$. By definition of the dual modules, (b) is equivalent to the following statement:
\begin{itemize}
\item[(b1)]  $\bigotimes_{i=1}^s V(f_i)$ is of lowest $\ell$-weight if and only if $a_i \neq b_j$ for $1 \leq i < j \leq s$.
\end{itemize}
We only prove (a). For $1 \leq i \leq s$, let $u_i^+$ (resp. $u_i^-$) be a highest (resp. lowest) $\ell$-weight vector in $V(f_i)$. Then from the explicit realization of $V(f_i)$ we see that
\begin{displaymath}
|u_i^+| = \even,\quad s_{11}(z) u_i^+ = \frac{1-za_i}{1-zb_i} u_i^+,\quad s_{22}(z) u_i^+ = u_i^+,\quad s_{21}(z) u_i^+ = \frac{z\lambda_i}{1-zb_i} u_i^-
\end{displaymath}
where $\lambda_i \in \BC^{\times}$. We remark that Lemma \ref{lem: cyclicity Chari} still holds when replacing $U_q(\Gaff)$-modules by $Y_q(\Glie)$-modules. Indeed, if $W$ is a highest $\ell$-weight $Y_q(\Glie)$-module with $w$ a highest $\ell$-weight vector, then by Theorem \ref{thm: Ding-Frenkel} we see that $W$ is spanned by vectors of the form $X_{1,n_1}^- \cdots X_{1,n_r}^- w$ where $r \in \BZ_{\geq 0}$ and $n_i \in \BZ_{\geq 1}$ for $1 \leq i \leq r$.

Let $V: = \bigotimes_{i=1}^s V(f_i)$ and $u := \bigotimes_{i=1}^s u_i^+$. Via the action of the $s_{ii}^{(0)}$, $V$ is $\BQ$-graded:
\begin{displaymath}
(V)_{\lambda}:= \{ x \in V | s_{ii}^{(0)} x = q^{(\epsilon_i,\lambda)} x \quad \textrm{for}\ i = 1,2  \}.
\end{displaymath}
As $|u_i^+|_{\BQ} = 0, |u_i^-|_{\BQ} = -\alpha_1$, we see that $(V)_{-\alpha_1}$ is generated by the vectors 
\begin{displaymath}
w_j := (\bigotimes_{i=1}^{j-1} u_i^+) \otimes u_i^- \otimes (\bigotimes_{j=i+1}^s u_j^+) \quad \textrm{for}\ 1 \leq j \leq s.
\end{displaymath} 

On the other hand, set $V' := Y_q(\Glie) u$. As a highest $\ell$-weight module, $V'$ is $\BQ$-homogeneous. Moreover, from Theorem \ref{thm: Ding-Frenkel} we see that 
\begin{displaymath}
(V')_{-\alpha_1} = \sum_{n \in \BZ_{\geq 1}} \BC X_{1,n}^- u = \sum_{n \in \BZ_{\geq 1}} \BC s_{21}^{(n)} u.
\end{displaymath}
In other words, $(V')_{-\alpha_1}$ is generated by the coefficients of $s_{21}(z) u \in V[[z]]$.

Suppose first that $V = V'$ is of highest $\ell$-weight. Then the coefficients of 
\begin{eqnarray*}
s_{21}(z) u &=& \sum_{i=1}^s (\prod_{j=1}^{i-1} s_{22}(z) u_j^+) \otimes s_{21}(z) u_i^+ \otimes (\bigotimes_{j=i+1}^s s_{11}(z) u_j^+) =  \frac{z \sum_{i=1}^s \lambda_i g_i(z) w_i}{\prod_{l=1}^s (1-zb_l)} 
\end{eqnarray*}
generate an $s$-dimensional subspace. Here the $g_i(z) \in \BC[z]$ are defined by
\begin{displaymath}
g_i(z) = \prod_{j=1}^{i-1}(1-zb_j) \prod_{j=i+1}^s (1-za_j).
\end{displaymath}  
It follows that the polynomials $g_i(z) \in \BC[z]$ must be linearly independent. In view of Lemma \ref{lem: linear independence of polynomials} below, we must have $b_i \neq a_j$ for $1 \leq i < j \leq s$, as desired.

Next suppose that $b_i \neq a_j$ for $1 \leq i < j \leq s$. We show by induction on $s$ that $V$ is of highest $\ell$-weight. For $s = 1$ this is evident. Assume $s > 1$. Then we can assume furthermore that $\bigotimes_{i=2}^s V(f_i)$ is of highest $\ell$-weight. Now Lemma \ref{lem: cyclicity Chari} says that
\begin{displaymath}
V = Y_q(\Glie) w_1.
\end{displaymath}
Since $b_i \neq a_j$ for $1 \leq i < j \leq s$, the polynomials $g_i(z)$ are linearly independent (Lemma \ref{lem: linear independence of polynomials}). Hence the coefficients of $s_{21}(z) u$ generate an $s$-dimensional subspace. It follows that $w_1 \in V'$. Hence $V = V'$ is of highest $\ell$-weight. This completes the proof of Theorem \ref{thm: web property}. \hfill $\Box$
\begin{lem}   \label{lem: linear independence of polynomials}
Let $k \in \BZ_{>0}$. Let $a_i,a_i' \in \BC$ be given for $1 \leq i \leq k$. For $1 \leq j \leq k$, define
\begin{displaymath}
f_j(z) := (\prod_{i=1}^{j-1}(1 - z a_i)) (\prod_{i=j+1}^k (1 - z a_i')) \in \BC[z].
\end{displaymath}
Then the  $f_j(z)$ are linearly independent if and only if $a_i \neq a_j'$ for all $1 \leq i < j \leq k$.
\end{lem}
\begin{proof}
The $k$ polynomials $f_j(z)$ are of degree $\leq k-1$. Introduce 
\begin{displaymath}
f_1(z) \wedge f_2(z) \wedge \cdots \wedge f_k(z) = \Delta (1 \wedge z \wedge \cdots \wedge z^{k-1}) \in \wedge^k \BC[z].
\end{displaymath}
Then the $f_j(z)$ are linearly independent if and only if $\Delta \neq 0$. For $j+s \leq k$, take
\begin{displaymath}
f_j^{(s)}(z) = (\prod_{i=1}^{j-1}(1-za_i)) (\prod_{i=j+s+1}^k (1 - z a_i')).
\end{displaymath}
Then $f_j^{(0)}(z) = f_j(z)$ and $f_i^{(s)}(z) - f_{i+1}^{(s)}(z) = (a_i - a_{i+s+1}') z f_i^{(s+1)}(z)$ for $i+s+1 \leq k$. Take $\omega = 1 \wedge z \wedge \cdots \wedge z^{k-1}$. We have
\begin{eqnarray*}
\Delta \omega &=& \bigwedge_{i=1}^k f_i^{(0)}(z) = (\bigwedge_{i=1}^{k-1} f_i^{(0)}(z)-f_{i+1}^{(0)}(z)) \wedge f_k(z) = (\bigwedge_{i=1}^{k-1} (a_i - a_{i+1}') z f_i^{(1)}(z)) \wedge f_k^{(0)}(z)\\
 &=& (\prod_{i=1}^{k-1}(a_{i+1}'-a_i)) 1 \wedge (\bigwedge_{i=1}^{k-2} z(f_i^{(1)}(z)-f_{i+1}^{(1)}(z))) \wedge z f_{k-1}^{(1)}(z)  \\
 &=& (\prod_{i=1}^{k-1}(a_{i+1}'-a_{i}))(\prod_{i=1}^{k-2} (a_{i+2}' - a_i)) 1 \wedge z \wedge (\bigwedge_{i=1}^{k-2} z^2 f_i^{(2)}(z)) \\
 &=& \cdots = \prod_{1 \leq i < j \leq k} (a_j' - a_i) \omega. 
\end{eqnarray*}
Clearly $\Delta \neq 0$ if and only if $a_i \neq a_j'$ for $1 \leq i < j \leq k$. 
\end{proof}
\section{Tensor products of Kirillov-Reshetikhin modules}  \label{sec: proof}
In this section, based on the representation theory of $Y_q(\mathfrak{gl}(1,1))$ and the lemma of Chari, we prove a cyclicity result on tensor products of Kirillov-Reshetikhin modules over $U_q(\Gaff)$. 

Thanks to Proposition \ref{prop: evaluation morphism}, there is a highest weight representation theory for the quantum superalgebra $U_q(\Glie)$. In particular, for $\lambda \in \BP$, we have simple $U_q(\Glie)$-module $(DF^{-1})^* L(\lambda)$ which will be written as $L(\lambda)$. It is the simple $U_q(\Glie)$-module generated by a vector $v_{\lambda}$: 
\begin{displaymath}
|v_{\lambda}| = |\lambda|,\quad s_{ij} v_{\lambda} = 0, \quad s_{kk} v_{\lambda} = q^{(\epsilon_k,\lambda)} v_{\lambda} \quad (i,j,k \in I, i < j).
\end{displaymath}
For $a \in \BC^{\times}$, define the evaluation morphism $\ev_a := \ev \circ \Phi_a: U_q(\Gaff) \longrightarrow U_q(\Glie)$, here $\ev$ and $\Phi_a$ are given by Proposition \ref{prop: evaluation morphism} and by Formula \eqref{for: automorphisms of Z-graded superalgebras} respectively. We can pull back $U_q(\Glie)$-modules $V$ to get $U_q(\Gaff)$-modules $\ev_a^* V$.
When there is no confusion, we simply write $v = \ev_a^* v$ for $v \in V$. For example, take $\lambda \in \BP$. Consider $\ev_a^* L(\lambda)$. Let $v_{\lambda}$ be a highest weight vector for the $U_q(\Glie)$-module $L(\lambda)$, then $v_{\lambda}$ is a highest $\ell$-weight vector 
\begin{displaymath}
|v_{\lambda}| = |\lambda|,\quad s_{ii}(z) v_{\lambda} = (q^{(\lambda,\epsilon_i)} - z a q^{-(\lambda,\epsilon_i)}) v_{\lambda},\quad t_{ii}(z) v_{\lambda} = (q^{-(\lambda,\epsilon_i)} - z^{-1}a^{-1}q^{(\lambda,\epsilon_i)}) v_{\lambda}.
\end{displaymath}
\begin{defi}   \label{def: Kirillov-Reshetikhin modules}
The $U_q(\Gaff)$-modules $\ev_a^*L(k\varpi_r) \otimes \BC_{|k\varpi_r|}$ for $a \in \BC^{\times},r \in I_0, k \in \BZ_{\geq 0}$, are called Kirillov-Reshetikhin modules, denoted by $W_{k,a}^{(r)}$.
\end{defi}
For $r \in I_0$, set $r^{\vee} = r$ if $r \leq M$ and $r^{\vee} = M+N-r$ otherwise. 
\begin{thm}   \label{thm: cyclicity of tensor product of KR modules}
Let $r \in I_0, k \in \BZ_{>0}$ and $a_1,a_2,\cdots,a_k \in \BC^{\times}$. Assume that 
\begin{displaymath}
\frac{a_i}{a_j} \notin \{ q_r^{-2t} \ |\ t = 1,2,\cdots,r^{\vee} \} \quad \textrm{for}\ 1 \leq i < j \leq k.
\end{displaymath}
Then the $U_q(\Gaff)$-module $\bigotimes_{j=1}^k W_{1,a_j}^{(r)}$ is of highest $\ell$-weight. 
\end{thm}
Note that we make the assumption $M,N >0$ throughout the paper. When $N = 0$, the above theorem is a special case of cyclicity results in \cite{AK,Chari,Kashiwara,VV}.

The proof uses reduction arguments as in \cite{Chari}. Let $A,B$ be Hopf superalgebras and let $h: A \longrightarrow B$ be a morphism of superalgebras which does NOT necessarily respect the coproducts. Let $V$ be a $B$-module. Suppose that $W$ is a sub-vector-superspace of $V$ stable by $h(A)$. Then the action of $h(A)$ endows $W$ with an $A$-module structure, denoted by $h^{\bullet} W$. (It is a sub-$A$-module of $h^*(V)$.) Now let $V_1,V_2$ be two $B$-modules and let $W_1,W_2$ be subspaces of $V_1,V_2$ such that $V_1,V_2,V_1 \otimes V_2$ are all stable by $h(A)$. Consider the two $A$-modules $h^{\bullet} W_1 \otimes h^{\bullet} W_2$ and $h^{\bullet}(W_1 \otimes W_2)$ with the same underlying vector superspace $W_1 \otimes W_2$. By definition, they are not necessarily isomorphic $A$-modules as the former uses coproduct of $A$ and the later $B$. In our case, the identity map will be an isomorphism of these two $A$-modules due to specific choices of the $V_i,W_i,A,B$.
\begin{proof}
It is enough to consider the case $1 \leq r \leq M$. Indeed, for $M + 1 \leq r < M+N$, Propositions \ref{prop: from gl(N,M) to gl(M,N)}-\ref{prop: evaluation morphism} indicate that $f^* W_{k,a}^{(r)} \cong W_{k,a}'^{(N+M-r)}$ is a Kirillov-Reshetikhin module over $U_q(\widehat{\mathfrak{gl}(N,M)})$, corresponding to the fundamental weight $\varpi_{N+M-r}$. (This is the reason for defining $r^{\vee}$.) From now on, assume $1 \leq r \leq M$.  We prove the theorem by induction on $r$. Let $U:=Y_q(\mathfrak{gl}(1,1))$ as in \S\ref{sec: Yangian} and $h: U \longrightarrow U_q(\Gaff)$ be the superalgebra morphism 
\begin{displaymath}
s_{11}(z) \mapsto s_{11}(z),\quad s_{12}(z) \mapsto s_{1,M+N}(z),\quad s_{21}(z) \mapsto s_{M+N,1}(z),\quad s_{22}(z) \mapsto s_{M+N,M+N}(z).
\end{displaymath}
Fix $a \in \BC^{\times}$. The $U_q(\Gaff)$-module $W_{1,a}^{(r)} = \ev_a^* L(\varpi_r)$ is $\BP$-graded under the action of the $s_{ii}^{(0)}$. By Theorem \ref{thm: BKK Schur-Weyl duality},  $(W_{1,a}^{(r)})_{\lambda}$ is non-zero if and only if 
\begin{displaymath}
\lambda = \epsilon_{i_1} + \epsilon_{i_2} + \cdots + \epsilon_{i_r}
\end{displaymath} 
where $1 \leq i_1 \leq i_2 \leq \cdots \leq i_r \leq M+N$ and $i_s < i_{s+1}$ if $i_s \leq M$. Moreover, for such $\lambda$, $(W_{1,a}^{(r)})_{\lambda}$ is always one-dimensional, and for $x \in (W_{1,a}^{(r)})_{\lambda}$,
\begin{displaymath}
s_{ii}(z) x = (q^{(\lambda,\epsilon_i)} - z a q^{-(\lambda,\epsilon_i)}) x,\quad t_{ii}(z) x = (q^{-(\lambda,\epsilon_i)} - z^{-1}a^{-1}q^{(\lambda,\epsilon_i)}) x.
\end{displaymath}
Let $v_a^+$ (resp. $v_a^-$) be a highest (resp. lowest) $\ell$-weight vector in $W_{1,a}^{(r)}$. Define
\begin{displaymath}
u_a^+ = s_{1,M+N}^{(0)} v_a^-,\quad u_a^- = t_{M+N,1}^{(0)} v_a^+.
\end{displaymath}
By Theorem \ref{thm: BKK Schur-Weyl duality}, $|v_a^-|_{\BP} = r \epsilon_{M+N}$, together with the following Chevalley relation 
\begin{displaymath}
s_{1,M+N}^{(0)} t_{M+N,1}^{(0)} + t_{M+N,1}^{(0)} s_{1,M+N}^{(0)} = (q-q^{-1}) (t_{11}^{(0)} s_{M+N,M+N}^{(0)} - s_{11}^{(0)} t_{M+N,M+N}^{(0)})
\end{displaymath}
we obtain $u_a^{\pm} \neq 0$ and 
\begin{displaymath}
 |u_a^+|_{\BP} = \epsilon_1 + (r-1)\epsilon_{M+N},\quad |u_a^-|_{\BP} =\epsilon_2 + \cdots + \epsilon_r + \epsilon_{M+N}.
\end{displaymath}
Introduce vector subspaces $K^+(a) = \BC v_a^+ + \BC u_a^-,\ K^-(a) = \BC v_a^- + \BC u_a^+ \subseteq W_{1,a}^{(r)}$. The $\BP$-grading on $W_{1,a}^{(r)}$ says that the subspaces $K^{\pm}(a)$ are both stable by $h(U)$. Let $h^{\bullet} K^{\pm}(a)$ be the $U$-modules thus obtained. The following observations are crucial:
\begin{itemize}
\item[(1)] if $s_{li}(z) K^+(a) \neq 0$ and $(i \in \{1,M+N\}, 1 < l < M+N)$, then $r < l < M+N$; 
\item[(2)] if $i \neq l$ and $r < l < M+N$, then $s_{il}(z) K^{\pm}(a) = 0$;
\item[(3)] $\bigotimes_{j=2}^k K^+(a_j)$ is a subspace of $\bigotimes_{j=2}^k W_{1,a_j}^{(r)}$ stable by $h(U)$;
\item[(4)] $K^-(a_1) \otimes (\bigotimes_{j=2}^k K^+(a_j))$ is a subspace of $\bigotimes_{j=1}^k W_{1,a_j}^{(r)}$ stable by $h(U)$;
\item[(5)] the identity map is an isomorphism of $U$-modules
\begin{displaymath}
h^{\bullet} (K^-(a_1) \otimes (\bigotimes_{j=2}^k K^+(a_j))) \cong h^{\bullet} K^-(a_1) \otimes (\bigotimes_{j=2}^k h^{\bullet}K^+(a_j)).
\end{displaymath}
\end{itemize} 
Here, (1)-(2) follow directly from the weight gradings. (3) and (4) can be deduced from (1)-(2) by induction on $k$. By (1)-(2) and Equation \eqref{for: coproduct for quantum affine superalgebra S}, the $(h^{\otimes k})\Delta_k(x) - \Delta_k(h(x)) \in U_q(\Gaff)^{\otimes k}$  with $x = s_{ij}^{(n)} \in U$ acts as zero on the subspace in (4). Here the first $\Delta_k$ is the $k$-fold coproduct for $U = Y_q(\mathfrak{gl}(1,1))$ and the second for $U_q(\Gaff)$; the $k$-fold coproducts $\Delta_k: A \longrightarrow A^{\otimes k}$ with $k \geq 2$ of a Hopf superalgebra $(A,\Delta,\varepsilon)$ are defined inductively $\Delta_k = (\Delta \otimes \Id_A^{\otimes k-2})\Delta_{k-1}$ and $\Delta_2 = \Delta$. This proves (5).  
Now, as $U$-modules, by using notations in \S \ref{sec: simple} 
\begin{align*}
& h^{\bullet} K^-(a) \cong \BC_{(r-1)\odd} \otimes \BC_{(q,q^{1-r})} \otimes \BC_{1-zaq^{2r-2}} \otimes V(\frac{1-zaq^{-2}}{1-zaq^{2r-2}}),   \\
& h^{\bullet} K^+(a) \cong   \BC_{(q,1)} \otimes \BC_{\frac{1}{1-za}} \otimes V(\frac{1 - zaq^{-2}}{1-za}).
\end{align*}
Since $a_1 \neq a_i q^{-2r}$ and $a_j \neq a_l q^{-2}$ for $1 < i \leq k$ and for $1<j<l\leq k$, by Theorem \ref{thm: web property}, the $U$-modules in (5) are of highest $\ell$-weight. In particular,
\begin{displaymath}
(6): v_{a_1}^- \otimes (\bigotimes_{i=2}^k v_{a_i}^+) \in h(U)(u_{a_1}^+ \otimes (\bigotimes_{i=2}^k v_{a_i}^+)) \subseteq U_q(\Gaff) (u_{a_1}^+ \otimes (\bigotimes_{i=2}^k v_{a_i}^+)).
\end{displaymath}

Now let us prove that $\otimes_{i=1}^k W_{1,a_i}^{(r)}$ is of highest $\ell$-weight by a second induction on $k$, as in the proof of Theorem \ref{thm: web property}. Assume that $\otimes_{i=2}^k W_{1,a_i}^{(r)}$ is of highest $\ell$-weight. By Lemma \ref{lem: cyclicity Chari} and by (6), it is enough to prove the following assertion:
\begin{displaymath}
(7):\quad u_{a_1}^+ \otimes (\bigotimes_{i=2}^k v_{a_i}^+) \in U_q(\Gaff) (\bigotimes_{i=1}^k v_{a_i}^+).
\end{displaymath}
If $r = 1$, then $u_{a_1}^+ \in \BC^{\times} v_{a_1}^+$ and (7) is trivial. If $r > 1$, let
$U':= U_q(\widehat{\mathfrak{gl}(M-1,N)})$ and let $g: U' \longrightarrow U_q(\Gaff)$ be the superalgebra morphism
\begin{displaymath}
s_{ij}(z) \mapsto s_{i+1,j+1}(z),\quad t_{ij}(z) \mapsto t_{i+1,j+1}(z).
\end{displaymath}
For $b \in \BC^{\times}$, let $K(b) = g(U') v_{b}^+ \subseteq W_{1,b}^{(r)}$ so that it is stable by $g(U')$. The weight grading on $W_{1,b}^{(r)}$ and Theorem \ref{thm: BKK Schur-Weyl duality} applied to $\mathfrak{gl}(M-1,N)$ and Dynkin node $r-1$ show that the induced $U'$-module $g^{\bullet} K(b)$ is isomorphic to the Kirillov-Reshetikhin module $W_{1,b}^{'(r-1)}$ over $U'$ and $u_b^+ \in K(b)$. The same arguments as (3)-(5) above indicate that: $\bigotimes_{j=1}^k K(a_j)$ is a subspace of $\bigotimes_{j=1}^k W_{1,a_j}^{(r)}$ stable by $g(U')$; as $U'$-modules
\begin{displaymath}
g^{\bullet} (\bigotimes_{j=1}^k K(a_j)) \cong \bigotimes_{j=1}^k g^{\bullet} K(a_j) \cong \bigotimes_{j=1}^k W_{1,a_j}^{'(r-1)}. 
\end{displaymath}
The first induction hypothesis on $r$ shows that the RHS above is of highest $\ell$-weight, from which follows (7). 
This concludes the proof of Theorem \ref{thm: cyclicity of tensor product of KR modules}. 
\end{proof}
% More general cyclicity results on tensor products of Kirillov-Reshetikhin modules of the form $\bigotimes_{j=1}^k W_{l_j,a_j}^{(r_j)}$ can be hopefully obtained in this way. For this purpose, it is necessary to determine first of all the zeros and poles of $R$-matrices between $W_{l_1,a_1}^{(r_1)}$ and $W_{l_2,a_2}^{(r_2)}$, in view of Kashiwara's cyclicity results in the non-graded case \cite{Kashiwara}. In type A, this should be possible after a fusion procedure \cite{Date, Poulain}.

Let us end this section with a detailed discussion of the modules $W_{1,a}^{(1)}$.  
There is a representation $\rho_{(1)}$ of $U_q(\Glie)$ on $\BV$: 
\begin{eqnarray*}
&&(\rho_{(1)} \otimes \Id_{\End \BV}) (T) = (\Id_{\End \BV} \otimes \tau) (R^{-1}), (\rho_{(1)} \otimes \Id_{\End \BV})(S) = (\Id_{\End \BV} \otimes \tau) ((R')^{-1}), \\
&& \rho_{(1)}(s_{ii}) = q_i E_{ii} + \sum_{j \neq i} E_{jj} = \rho_{(1)} (t_{ii}^{-1}) \quad (\textrm{for}\ i \in I),   \\
&& \rho_{(1)}(s_{ij}) = (q_i - q_i^{-1}) E_{ij},\quad \rho_{(1)}(t_{ji}) = (q_i^{-1}-q_i) E_{ji}\quad (\textrm{for}\ 1 \leq i < j \leq M+N).
\end{eqnarray*}
By Proposition \ref{prop: evaluation morphism} and Example \ref{example: natural representation quantum superalgebra}, $\rho_{(0)} = \rho_{(1)} \circ DF$ and $\BV \cong L(\epsilon_1)$ as  $U_q(\Glie)$-modules. For $a \in \BC^{\times}$, the representation $\rho_a := \rho_{(1)} \circ \ev_a$ corresponds to the $U_q(\Gaff)$-module $W_{1,a}^{(1)}$.

The following lemma says that Perk-Schultz $R$-matrices can be interpreted as intertwining operators, from which comes naturally the Yang-Baxter equation in \S \ref{sec: Perk-Schultz matrix}.
\begin{lem}
Let $a, b \in \BC^{\times}$. Then $c_{\BV,\BV} \circ R(z,w)|_{(z,w) = (a,b)}: W_{1,a}^{(1)} \otimes W_{1,b}^{(1)} \longrightarrow W_{1,b}^{(1)} \otimes W_{1,a}^{(1)}$ is a morphism of $U_q(\Gaff)$-modules. 
\end{lem} 
The proof is direct, using symmetry properties of the Perk-Schultz $R$-matrix. 

\begin{prop}   \label{prop: tensor product of natural representations}
Let $k \in \BZ_{>0}$ and $a_i \in \BC^{\times}$ for $1 \leq i \leq k$. The $U_q(\Gaff)$-module $\otimes_{i=1}^k W_{1,a_i}^{(1)}$ is of highest $\ell$-weight if and only if $a_i \neq a_j q^{-2}$ for $1 \leq i < j \leq k$. It is of lowest $\ell$-weight if and only if $a_i \neq a_j q^{2}$ for $1 \leq i < j \leq k$.  
\end{prop}
The proof of this proposition is again quite similar to that of Theorem \ref{thm: web property}. We remark that the method of proof of Theorem \ref{thm: cyclicity of tensor product of KR modules} however does not apply directly to obtain similar sufficient conditions for a tensor product of the $W_{1,a}^{(r)}$ to be of lowest $\ell$-weight.

\appendix
\section{Proof of the coproduct formulas} \label{sec: app}
Proposition \ref{prop: Drinfeld coproduct estimation} is proved in essentially the same way as \cite[Prop.5.4]{Z}. However, it should be noted that the coproduct estimations in {\it loc. cit} are not enough in the present paper as seen from the proof of Chari's Lemma \ref{lem: cyclicity Chari}.

Without loss of generality, we shall prove the coproduct formulas for $K_{j,s}^+, X_{i,n}^{\pm}$ for $i \in I_0, j \in I$ and $s,n \in \BZ_{\geq 0}$. Proof of other cases is parallel. For simplicity, let $U:= U_q(\Gaff)$. In the following, for  two vectors $x,y$ in a vector space, we write $x \doteq y$ if $x \in \BC^{\times }y$.

As a first step, we express the $h_{i,1}$ as quantum brackets. 
Let $x \in U_{\alpha}, y \in U_{\beta}$ be $\BQ$-homogeneous. Define 
\begin{displaymath}
\lfloor x, y \rfloor := x y - (-1)^{|\alpha||\beta|} q^{(\alpha,\beta)} y x.
\end{displaymath}
Given $x_s \in U_{\beta_s}$ for $1 \leq s \leq r$, define {\it iterated quantum brackets}
\begin{displaymath}
\lfloor x_1,x_2,\cdots,x_r \rfloor_L := \lfloor \lfloor x_1,x_2,\cdots,x_{r-1} \rfloor_L, x_r \rfloor,\quad \lfloor x_1,x_2,\cdots,x_r \rfloor_R := \lfloor x_1,\lfloor x_2,\cdots,x_{r-1}\rfloor_R \rfloor.
\end{displaymath}
\begin{lem}   \label{lem: zero node}
 $\lfloor X_{1,1}^-, X_{2,0}^-, X_{3,0}^-, \cdots, X_{M+N-1,0}^- \rfloor_L \doteq s_{M+N,1}^{(1)} (s_{11}^{(0)})^{-1}$. 
\end{lem}
\begin{proof}
Fix $i,j,k \in I$ such that $i < j < k$. By taking the matrix coefficients of $v_j \otimes v_i \mapsto v_{k} \otimes v_j$ in Equation \eqref{rel: RTS = STR}
we obtain
\begin{eqnarray*}
&\ & (-1)^{|k|(|i|+|j|)} (z-w) t_{kj}(z) s_{ji}(w) + z(q_j-q_j^{-1}) (-1)^{|j|(|k|+|i|)} t_{jj}(z) s_{ki}(w) \\
& = & (-1)^{|j|(|i|+|j|)} (z-w) s_{ji}(w)t_{kj}(z) + w(q_j-q_j^{-1})(-1)^{|i|(|i|+|j|)} s_{jj}(w) t_{ki}(z).
\end{eqnarray*}
Next by comparing the coefficients of $zw$ we get 
\begin{equation}   \label{equ: relations in qas}
[s_{ji}^{(1)}, t_{kj}^{(0)}]  = (q_j-q_j^{-1}) t_{jj}^{(0)} s_{ki}^{(1)}.
\end{equation}
Note that for $1 \leq j \leq M+N-1$  we have
\begin{displaymath}
X_{1,1}^- = -s_{21}^{(1)}(s_{11}^{(0)})^{-1}, \quad X_{j,0}^- = t_{j+1,j}^{(0)} (t_{jj}^{(0)})^{-1}.
\end{displaymath}
By repeatedly applying Equation \eqref{equ: relations in qas} we find the desired quantum bracket.
\end{proof} 
Let us introduce the $H_{i,s}$ for $i \in I$ and $s \in \BZ_{>0}$ by the following functional equations:
\begin{displaymath}
K_i^+(z) = s_{ii}^{(0)} \exp ((q_i-q_i^{-1}) \sum_{s \in \BZ_{>0}} H_{i,s} z^s) \in U[[z]].
\end{displaymath}
From Theorem \ref{thm: Ding-Frenkel} we deduce that for $i \in I, j \in I_0$
\begin{align*}
& [H_{i,s}, X_{j,n}^{\pm}] = 0 \quad \textrm{if}\ i \neq j, j+1,  \\
& [H_{i,s}, X_{i,n}^{\pm}] = \pm q_i^s \frac{[s]}{s}  X_{i,n+s}^{\pm}, \quad   [H_{i,s}, X_{i-1,n}^{\pm}] = \mp q_i^{-s} \frac{[s]}{s} X_{i,n+s}^{\pm}.
\end{align*}
Set $h_{i,s} := d_i H_{i,s} - d_{i+1} H_{i+1,s}$ for $i \in I_0, s \in \BZ_{>0}$. One can find  $c_{i} \in \BC^{\times}$ for $i \in I_0$ such that
\begin{displaymath}
[h_{i,1}, X_{i+1,n}^{\pm}] = \pm c_{i+1} X_{i+1,n+1}^{\pm},\quad [H_{1,1}, X_{1,n}^{\pm}] = \pm c_1 X_{1,n+1}^{\pm}
\end{displaymath}

Let us introduce $L_i := s_{ii}^{(0)} (s_{i+1,i+1}^{(0)})^{-1}$ for $i \in I_0$. Set
\begin{displaymath}
E_0 := s_{M+N,1}^{(1)} (s_{M+N,M+N}^{(0)})^{-1}, \quad E_i := X_{i,0}^+,\quad L_0 := (L_1L_2\cdots L_{M+N-1})^{-1}.
\end{displaymath}
The following lemma is proved in the same way as in \cite[Appendix 2]{Z}.
\begin{lem}   \label{lem: quantum brackets for degree 1 cartan elements}
$h_{i,1} \doteq \lfloor E_i, E_{i-1}, E_{i-2}, \cdots, E_1, E_{i+1}, E_{i+2}, \cdots, E_{M+N-1}, E_0 \rfloor_R$ for $i \in I_0$.
\end{lem}
As a second step, we compute the $\Delta(h_{i,1})$.
Introduce the length function $\ell: \BQ_{\geq 0} \longrightarrow \BZ_{\geq 0}$:
\begin{displaymath}
 \ell(\sum_{i\in I_0} n_i \alpha_i) = \sum_{i\in I_0} n_i.
 \end{displaymath}
In the following, when we write $\ell(\alpha)$,  it should be understood implicitly that $\alpha \in \BQ_{\geq 0}$. For $i \in I_0$, let $U_i$ be the subalgebra of $U$ generated by the $E_j$ with $j \in I_0 \setminus \{i\}$. Clearly $U_i$ is a $\BQ$-graded subalgebra. Let us first consider $h_{1,1}$:
\begin{displaymath}
h_{1,1} \doteq \lfloor E_1,E_2,\cdots, E_{M+N-1}, E_0 \rfloor_R.
\end{displaymath}
To compute $\Delta (h_{1,1})$, notice first that 
\begin{displaymath}
\lfloor \Delta E_1, \Delta E_2, \cdots, \Delta E_{M+N-1}, E_0 \otimes L_0^{-1} \rfloor_R = \lfloor E_1,E_2,\cdots, E_{M+N-1},E_0 \rfloor_R \otimes 1.
\end{displaymath}
Note that $|E_0|_Q = -(\alpha_1+\alpha_2+\cdots +\alpha_{M+N-1})$. It follows that 
\begin{eqnarray*}
\Delta(h_{1,1}) &\in & \sum_{i\in I_0} \BC^{\times} \lfloor 1 \otimes E_1,\cdots, 1 \otimes E_{i-1}, E_i \otimes L_i^{-1}, 1 \otimes E_{i+1}, \cdots, 1 \otimes E_{M+N-1}, 1 \otimes E_0 \rfloor_R  \\  
&\ & + 1 \otimes h_{1,1} + h_{1,1} \otimes 1 + \sum_{\ell(\alpha) > 1} U_{\alpha} \otimes U_{-\alpha}.
\end{eqnarray*}
Set $r_{1,i}$ equal to
\begin{displaymath}
E_i \otimes   \lfloor E_1,E_2,\cdots,E_{i-1},E_{i+1},\cdots, E_{M+N-1},E_0 \rfloor_R L_i^{-1}.
\end{displaymath}
Then by definition of quantum brackets $r_{1,i} = 0$ for $i \neq 1,2$, and 
\begin{align*}
& r_{1,1} \doteq E_1 \otimes X_{1,1}^-,   \quad  r_{1,2} \doteq E_2 \otimes \lfloor E_1, \lfloor X_{1,1}^-, X_{2,0}^-\rfloor L_1L_2 \rfloor L_2^{-1} \doteq  E_2 \otimes X_{2,1}^-.
\end{align*}
In other words, we have 
\begin{eqnarray*}  
\Delta (h_{1,1}) &\in & h_{1,1} \otimes 1 + 1 \otimes h_{1,1} +  \sum_{i\in I_0: i \leq 2} \BC^{\times} E_i \otimes X_{i,1}^-  \\
&\ & + \sum_{\ell(\alpha) > 1} (U_1)_{\alpha}  \otimes U_{-\alpha} + \sum_{\ell(\alpha-\alpha_1) > 0} U_{\alpha} \otimes U_{-\alpha}.
\end{eqnarray*}
Similar arguments applied to the $h_{i,1}$ lead to the following coproduct formulas.
\begin{lem} \label{lem: coproduct for degree 1 cartan elements}
Let $1 \leq i \leq M+N-1$. Then for all $s \in I_0$ we have
\begin{eqnarray*}  
\Delta (h_{i,1}) & \in &   h_{i,1} \otimes 1 + 1 \otimes h_{i,1} + \sum_{j \in I_0: i-1\leq j \leq i+1} \BC^{\times} E_j \otimes X_{j,1}^-   \\
&\ & + \sum_{\ell(\alpha) > 1} (U_s)_{\alpha} \otimes U_{-\alpha} + \sum_{\ell(\alpha - \alpha_s) > 0} U_{\alpha} \otimes U_{-\alpha},  \\
\Delta (H_{1,1}) & \in & H_{1,1} \otimes 1 + 1 \otimes H_{1,1} + \BC^{\times} E_1 \otimes X_{1,1}^- + \sum_{\ell(\alpha-\alpha_1)>1} U_{\alpha} \otimes U_{-\alpha}.
\end{eqnarray*}
\end{lem}
For the second part of the lemma, by Gauss decomposition $K_1^+(z) = H_1(z)$ we have $H_{1,1} = \frac{1}{q-q^{-1}} (s_{11}^{(0)})^{-1}s_{11}^{(1)}$. The coproduct estimation $\Delta(H_{1,1})$ comes from the explicit formula $\Delta(s_{11}^{(1)})$ and from the identities $E_1 = (s_{11}^{(0)})^{-1}s_{12}^{(0)}$ and $X_{1,1}^- = - s_{21}^{(1)}(s_{11}^{(0)})^{-1}$. 

The third step is to estimate the $\Delta(X_{i,n}^+)$ with $n \geq 0$. For $i \in I_0$, introduce
\begin{displaymath}
\sum_{n \geq 0} \Psi_{i,n} z^n := K_{i+1}^+(z)K_i^+(z)^{-1} \in U[[z]]
\end{displaymath}
 Let $A_i$ be the subalgebra of $U$ generated by the $L_i,\Psi_{i,n}$ for $n \in \BZ_{\geq 0}$. 
\begin{lem} \label{lem: coproduct for higher degree positive root vectors}
For $i \in I_0$ and $n \in \BZ_{\geq 0}$, we have
\begin{equation}  \label{equ: coproduct for X+}
\Delta (X_{i,n}^+) - 1 \otimes X_{i,n}^+ \in \sum_{m=0}^n X_{i,m}^+ \otimes A_i + \sum_{\ell(\alpha) > 1} (U_i)_{\alpha} \otimes U_{\alpha_i-\alpha} + \sum_{\ell(\alpha - \alpha_i)  > 0} U_{\alpha} \otimes U_{\alpha_i-\alpha}.
\end{equation}
\end{lem}
\begin{proof}
Let us assume first $2 \leq i \leq M+N-1$. Then  $[h_{i-1,1}, X_{i,n}^+] = c_i X_{i,n+1}^+$ for $n \in \BZ_{\geq 0}$. We prove the above coproduct formula by induction on $n \in \BZ_{\geq 0}$. Clearly,
\begin{displaymath}
\Delta(X_{i,0}^+) = 1 \otimes X_{i,0}^+ + X_{i,0}^+ \otimes L_i^{-1} \in 1 \otimes X_{i,0}^+ + X_{i,0}^+ \otimes A_i.
\end{displaymath}
Assume that the coproduct formula \eqref{equ: coproduct for X+} is true for $n$. Remark that for $j \in I_0 \setminus \{i\}$
\begin{displaymath}
[E_j \otimes X_{j,1}^-, 1 \otimes X_{i,0}^+] = 0.
\end{displaymath}
Lemma \ref{lem: coproduct for degree 1 cartan elements} applied to $h_{i-1,1}$ with $s = i$ gives  
\begin{eqnarray*}
\Delta (c_i X_{i,n+1}^+) &\in &  c_i 1\otimes X_{i,n+1}^+ +  \BC^{\times} E_i \otimes [X_{i,n}^+, X_{i,1}^-] \\
&\ & + \sum_{m=0}^n [h_{i-1,1}, X_{i,m}^+] \otimes A_i + \sum_{\ell(\alpha)>1} (U_i)_{\alpha} \otimes U_{\alpha_i-\alpha} + \sum_{\ell(\alpha-\alpha_i) > 0} U_{\alpha} \otimes U_{\alpha_i-\alpha}  \\
&\subseteq & c_i 1 \otimes X_{i,n+1}^+ + \sum_{m=0}^{n+1} X_{i,m}^+ \otimes A_i \\
&\ & + \sum_{\ell(\alpha) > 1} (U_i)_{\alpha} \otimes U_{\alpha_i-\alpha} + \sum_{\ell(\alpha-\alpha_i) > 0} U_{\alpha} \otimes U_{\alpha_i-\alpha}.
\end{eqnarray*}
This establishes Equation \eqref{equ: coproduct for X+}. 

Next, when $i = 1$,  we use the relation $[H_{i,1}, X_{1,n}^{\pm}] = \pm c_1 X_{1,n+1}^{\pm}$ and the coproduct formula $\Delta(H_{1,1})$ in Lemma \ref{lem: coproduct for degree 1 cartan elements}.
The rest is parallel as in the case $i > 1$.
\end{proof}
Lemma \ref{lem: coproduct for higher degree positive root vectors} can be viewed as a refinement of Equation \eqref{equ: coproduct X}. In a similar way, it is not difficult to prove Equation \eqref{equ: coproduct Y} by using Lemma \ref{lem: coproduct for degree 1 cartan elements}.

As the final step, let us prove Equation \eqref{equ: coproduct H}. 
\begin{cor}  \label{cor: first estimation of coproduct for H}
For $i \in I_0$ and $n \in \BZ_{\geq 0}$ we have
\begin{equation}   \label{equ: coproduct for higher degree cartan elements}
\Delta (\Psi_{i,n}) \in A_i \otimes A_i + \sum_{\ell(\alpha) > 0} U_{\alpha} \otimes U_{-\alpha}.
\end{equation}
\end{cor}
\begin{proof}
For $n = 0$, this is clear since $\Psi_{i,0} = L_i^{-1}$. For $n > 0$, by Theorem \ref{thm: Ding-Frenkel} we have 
\begin{displaymath}
\Psi_{i,n} = (q_i-q_i^{-1})^{-1} [X_{i,n}^+, X_{i,0}^-].
\end{displaymath}
It is enough to consider the bracket $[\Delta(X_{i,n}^+), \Delta (X_{i,0}^+)]$. Remark that by definition $X_{i,0}^-$ commutes with $U_i$.
Now Equation \eqref{equ: coproduct for higher degree cartan elements} follows from Equation \eqref{equ: coproduct for X+}.
\end{proof}
 To prove Equation \eqref{equ: coproduct H}, in view of the explicit coproduct formula for $K_1^+(z) = s_{11}(z)$ in Equation \eqref{for: coproduct for quantum affine superalgebra S}, it is enough to show that: for $i \in I_0$ and $n \in \BZ_{\geq 0}$
\begin{displaymath}
\Delta (\Psi_{i,n}) \in \sum_{m=0}^n \Psi_{i,m} \otimes \Psi_{i,n-m}  + \sum_{\ell(\alpha) > 0} U_{\alpha} \otimes U_{-\alpha}.
\end{displaymath}
Clearly, $\Delta (\Psi_{i,0}) = \Psi_{i,0} \otimes \Psi_{i,0}$. In view of Corollary \ref{cor: first estimation of coproduct for H} let us define $\Delta_i(\Psi_{i,n}) \in A_i \otimes A_i$ to be such that $\Delta(\Psi_{i,n}) - \Delta_i(\Psi_{i,n}) \in \sum_{\ell(\alpha) > 0} U_{\alpha} \otimes U_{-\alpha}$. From the highest $\ell$-weight representation theory (\S \ref{sec: highest l weight}) of the quantum affine superalgebra $U$ we observe that the subalgebra $A_i$ is an algebra of Laurent polynomials:
\begin{displaymath}
A_i = \BC[\Psi_{i,n} : n \in \BZ_{> 0}][\Psi_{i,0},\Psi_{i,0}^{-1}].
\end{displaymath}
So is the tensor algebra $A_i \otimes A_i$. It follows that an element $x \in A_i \otimes A_i$ is completely determined by the data $\chi \times \mu (x)$ where $\chi, \mu$ are algebra homomorphisms $A_i \longrightarrow \BC$. Since $\sum_{\ell(\alpha) > 0} U_{\alpha} \otimes U_{-\alpha}$ always annihilates the tensor product of two highest $\ell$-weight vectors, 
\begin{displaymath}
\chi \times \mu (\Delta_i(\Psi_{i,n})) = \chi \times \mu(\Delta(\Psi_{i,n}))  = \sum_{m=0}^n \chi(\Psi_{i,m}) \mu(\Psi_{i,n-m}) = \chi \times \mu (\sum_{m=0}^n \Psi_{i,m} \otimes \Psi_{i,n-m}).
\end{displaymath}
It follows that $\Delta_i(\Psi_{i,n}) = \sum_{m=0}^n \Psi_{i,m} \otimes \Psi_{i,n-m}$. \hfill $\Box$

\end{document}